\documentclass{amsart}
\usepackage{amsmath}
\usepackage{graphics,graphicx,psfrag, bbm}
\usepackage[applemac]{inputenc}
 
\newcommand{\vip}{\vskip0.15cm}

\newcommand{\indiq}{{\rm 1 \hskip-4pt 1}}

\newcommand{\nn}{\mathbb{N}}
\newcommand{\zz}{\mathbb{Z}}
\newcommand{\rr}{{\mathbb{R}}}

\newcommand{\cI}{{\mathcal I}}
\newcommand{\cJ}{{\mathcal G}}

\newcommand{\cS}{{\mathcal S}}
\newcommand{\cA}{{\mathcal A}}

\newcommand{\cC}{{\mathcal C}}
\newcommand{\cO}{{\mathcal O}}
\newcommand{\cP}{{\mathcal P}}
\newcommand{\cR}{{\mathcal R}}
\newcommand{\cZ}{{\mathcal Z}}

\newcommand{\adm}{\cA}
\newcommand{\conf}{\cC}
\newcommand{\cercle}{{\mathbb{S}^1}}
\newcommand{\tore}{{\mathbb{T}^2}}

\newcommand{\carre}{\Box}
\newcommand{\bfz}{{\bf z}}
\newcommand{\bfx}{{\bf x}}
\newcommand{\bfh}{{\bf h}}

\newcommand{\bill}{{\Phi}}

\usepackage{marvosym,fancybox}

\newtheorem{theorem}{\indent Th\'eor\`eme}[section]
\newtheorem{proposition}[theorem]{\indent Proposition}
\newtheorem{conjecture}[theorem]{\indent Conjecture}
\newtheorem{remark}[theorem]{\indent Remarque}
\newtheorem{lemma}[theorem]{\indent Lemme}
\newtheorem{definition}[theorem]{\indent D\'efinition}
\newtheorem{cor}[theorem]{\indent Corollaire}

\renewenvironment{proof}{\vip \noindent {\it Preuve. }}{\hfill$\square$ \vip}

\begin{document}

\title{Casse-Briques}\author{Xavier Bressaud, Marie-Claire Fournier}
\date{\today}

\begin{abstract}Cet article propose une version mathématique du jeu éponyme. Son étude systématique se révèle d'une complexité surprenante. Après un survol des propriétés générales du modèle, nous étudions plus précisément un cas très particulier puis mettons en évidence les difficultés combinatoires qui apparaissent dans un autre. 
\end{abstract}

\maketitle


\section{Introduction}

Nous nous proposons d'étudier un modèle mathématique inspiré de l'un des premiers  jeux video publiés, le casse-briques --- en anglais, {\it breakout} --- apparu dans les années 70.  Sur la partie supérieure de l'écran, un mur de briques ; en bas, un palet mobile sur un axe horizontal~; entre les deux, une bille (carrée, sur les premières versions) effectuant des allers-retours entre le palet sur lequel elle peut rebondir s'il est bien positionné  et les briques du mur sur lesquelles elle rebondit aussi, mais en {\it effaçant} à chaque impact la brique heurtée. Petit à petit, les briques disparaissent.  L'enjeu pour le joueur est de placer correctement le palet pour assurer le rebond : si la bille atteint le bas de l'écran, la partie est perdue.  Le jeu (ou du moins le tableau) s'arrête lorsque toutes les briques ont été atteintes et donc effacées. Dans les versions élaborées, le joueur peut avoir une influence sur l'angle avec lequel la bille rebondit sur le palet en impulsant une vitesse, en donnant une orientation. Bien des variantes ont suivi : disposition des briques, formes des briques, briques spéciales, multiplication des billes, obstacles inamovibles, qui pourront nous emmener plus loin. 


\begin{figure}
\includegraphics[width=60mm]{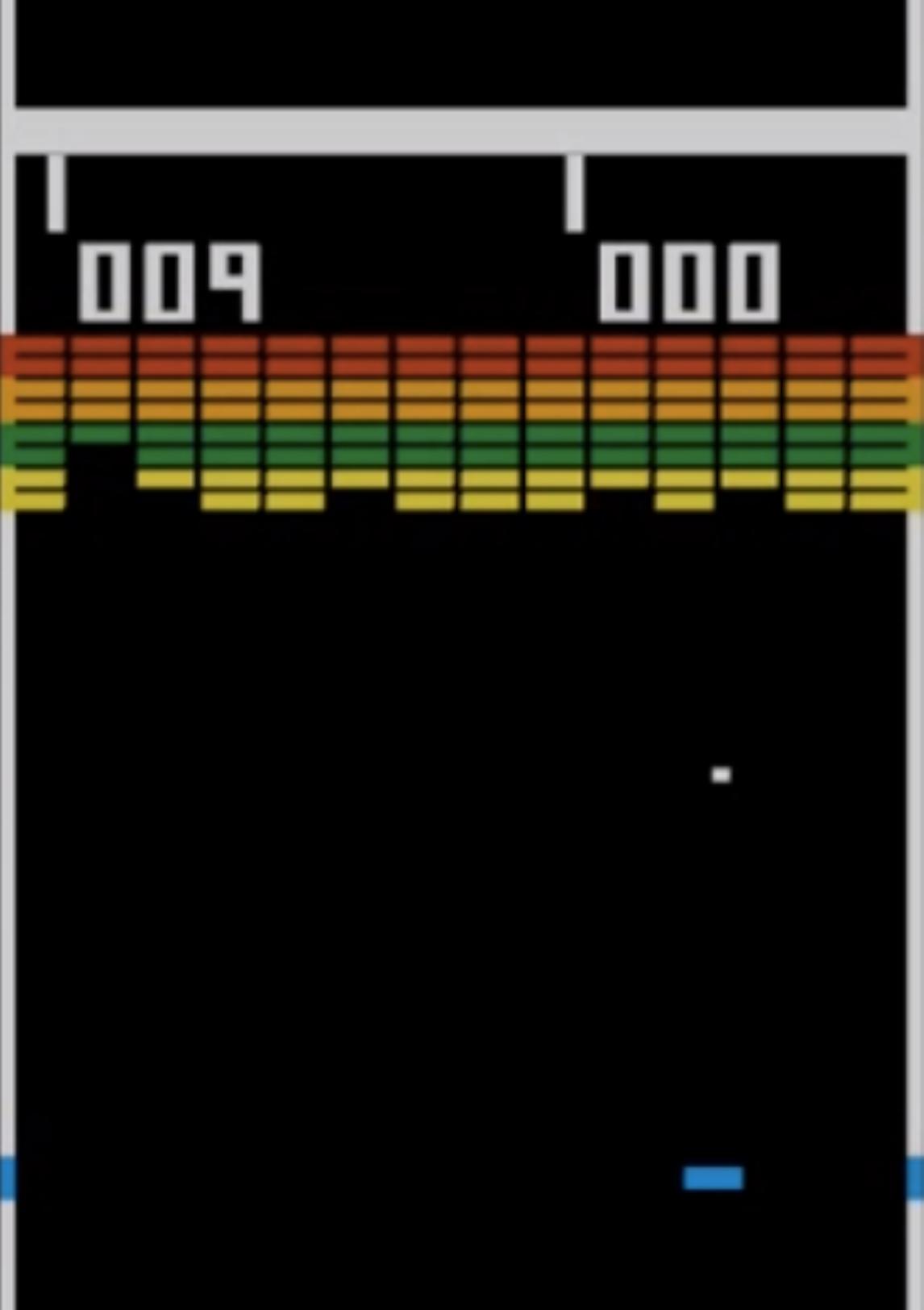}

  \caption{Une version  du jeu original {\it breakout}. D'après {\sc Wikipédia} : "{\it Breakout est initialement publié comme borne d'arcade en 1976. Celle-ci possède un écran monochrome sur lequel sont collées des bandes de plastique transparent colorées à l'emplacement des briques afin de les faire apparaître en couleur.
La borne d'arcade dessinée par Nolan Bushnell a été conçue à la base par Steve Jobs et Steve Wozniak, les fondateurs ultérieurs d'Apple. [...] 
Le jeu d'arcade fut un succès. En 1978, la version arcade originale de Breakout a été officiellement portée sur plusieurs systèmes, telles la Video Pinball, l'Atari 5200 et l'Atari 2600. Le jeu a connu depuis un nombre élevé de clones, le plus connu étant probablement Arkanoid de Taito.}" }
  \label{fig:breakout}
\end{figure}
Nous ne souhaitons pas nous préoccuper des mouvements du joueur ; supposons qu'il place toujours le palet au bon endroit et ne perde pas. Quel est alors le mouvement de la bille~? Combien de temps lui faut-il pour détruire toutes les briques ? Creuse-t-elle en priorité certaines directions ? Peut-il arriver que la bille se perde dans le mur, détruisant beaucoup de briques sans redescendre ? C'est en cherchant à formaliser ce type de questions que nous avons défini le modèle proposé. Naturellement, l'objet mathématique obtenu se révèle plus riche que le système modélisé  et nous finissons par l'étudier pour lui-même.

\vip

L'objet mathématique que nous allons présenter comme modèle du casse-briques est à penser comme une version élaborée d'un billard mathématique dont les bords changent lorsque la bille heurte certains bords. On se donne un domaine (polygonal mais pas nécessairement borné, éventuellement le plan tout entier) du plan dans lequel sont placés des obstacles (polygonaux, compacts). Un point matériel évolue à vitesse constante dans ce domaine (privé des obstacles) jusqu'à ce qu'il en atteigne la frontière ; si cette frontière est un bord du domaine, alors il rebondit (comme dans un billard classique) ; si c'est le bord d'un obstacle, il rebondit suivant les même règles, {\it mais} l'obstacle disparait. Le point matériel continue donc d'évoluer, mais maintenant dans un billard plus grand. 

\vip

Ce système dynamique n'est en général pas conservatif : le nombre d'obstacles ne peut que décroître. Il n'est donc pas question pour  l'étudier d'utiliser directement les outils standards de la théorie ergodique. Une autre particularité de ce système est que le flot ainsi défini n'est pas bijectif ; c'est un semi-flot. En effet, si on essaye de revenir en arrière après avoir heurté une brique, deux alternatives se présentent ; soit la brique (qu'on vient d'heurter et de faire disparaitre) était présente ; soit elle n'était pas présente et dans ce cas, la bille venait de traverser la case en question (et n'a pas rebondi). Ainsi, à certaines traversées de coté du pavage initial, on a (localement) deux préimages possibles. Réciproquement, il y a des états admissibles sans passé : la bille ne peut-être issue d'un obstacle encore présent.

\vip

Nous nous concentrons ici sur deux exemples pour lesquels les obstacles sont des carrés unité. 
malgré la simplicité de ces exemples, nous nous heurtons très vite à d'importantes difficultés combinatoires qui nous ont dissuadé de chercher à aborder le problème en toute généralité. Mais l'introduction de ce système ouvre des perspectives plus larges. Aussi nous commençons  par donner une définition formelle un peu générale d'un tel système et nous assurons qu'il satisfait quelques proprietés élémentaires. Nous nous consacrons ensuite à l'étude du casse-briques dans un domaine évoquant celui du jeu (avec un palet couvrant toute la largeur) mais qui ne serait pas borné en hauteur : nous nous demandons à quelle vitesse le mur peut être détruit. Pour certaines valeurs des paramètres, nous parvenons à nous ramener à l'étude de systèmes dynamiques classiques, nous mettant en position d'utiliser des résultats standards sur les billards ou les isométries par morceaux pour obtenir des résultats précis. Enfin, nous nous efforçons de donner une idée de la complexité combinatoire du casse-briques  dans un domaine plan  en présentant les quelques exemples que nous avons pu comprendre.

\subsection{Définition et existence}
Un {\em domaine} est une partie polygonale de $\rr^2$, disons une intersection (finie) de demi plans. Un {\em obstacle} est un domaine compact.  On se donne un domaine $\Delta$ (pas necessairement compact, possiblement $\rr^2$ tout entier) et une famille  d'obstacles $\cO = \{B_z ; z \in \cI\}$, indexée par un ensemble $\cI$ dénombrable, essentiellement disjoints (ne s'intersectant que sur leur frontière). Pour simplifier, nous supposerons que les mesures des angles et des cotés des polygones sont uniformément minorés.  On définit de manière classique le flot $(\Phi^{\cO}_t)_{t \in \rr}$ du billard dans $\Delta^\cO = \overline{\Delta \setminus \cup_{B \in \cO} B}$. Il s'agit d'un billard polygonal, éventuellement non borné.  C'est une famille d'applications du fibré unitaire tangent  $\Phi^{\cO}_t :T\Delta^\cO = \Delta^\cO \times \cercle \to T \Delta^\cO$ formant essentiellement un groupe.  Il convient de préciser ce qui se passe lorsqu'un point appartient à la frontière de $\Delta^\cO$~: on peut imposer que le vecteur vitesse (unitaire) pointe vers l'interieur du domaine ($T^+\Delta^\cO$). Les trajectoires du billard sont  continues à droite et limitées à gauche.

On appelle {\em configuration}, notée $\eta$, tout $\eta \in \{0,1\}^\cI$ qu'on peut voir comme une partie de l'ensemble des indices $\cI$  dénombrant $\cO$.  A une configuration $\eta$ on associe  une partie de $\cO$ en posant $\cO(\eta)  = \{B_z ; z \in \cI, \eta(z) =1\}$~; nous appelerons parfois abusivement cette liste aussi configuration.  
Etant donné un obstacle $P \in \cO(\eta)$, on note $\eta^{P}$ la configuration $\eta$ privée de l'obstacle $P$~:  $\cO(\eta^{P}) = \cO(\eta) \setminus\{P\}$. On définit l'ensemble des états admissibles comme la réunion sur l'ensemble des configurations $\eta$ des ensembles $T^+\Delta^\eta$ des états admissibles  pour le billard sur $\Delta^\eta$. Etant donnés une configuration $\eta$,  un point et une vitesse $(x,v) \in T^+\Delta^\eta$, on définit $\tau$ comme le premier instant où la bille rencontre le bord du domaine (et donc rebondit), 
$$\tau =  \tau^{\eta}(x,v) = \inf\{ t \geq 0 \, : \, \Phi^{\eta}_t(x,v) \in  \cup_{P \in \cO(\eta)} P\}$$
et on note $\beta = \beta^\eta(x,v)$ l'obstacle auquel appartient $\Phi^{\cO(\eta)}_\tau(x,v)$.    On définit maintenant un semi-flot $(\Phi_t)_{t \in \rr}$ sur $\Delta  \times \cercle \times \cP(\cO)$ en posant~:   
$$\Phi_t(x,v,\eta) = \left\{ \begin{array}{ll} (\Phi^\eta_t(x,v),\eta) & \hbox{ si } 0 \leq t <  \tau^\eta(x,v)\\  (\Phi^\eta_{\tau}(x,v),\eta^\beta) & \hbox{ si } t=\tau^\eta(x,v).\end{array} \right. $$
Cela définit bien un semi-flot pour tout $t>0$. En effet, on peut prolonger toutes les trajectoires exceptées celles qui atteignent une {\it singularité} (i.e. l'un des sommets de l'un des polygones). Remarquons que la dynamique ne dépend  que de la configuration dans un voisinage du point. L'ensemble des singularités est dénombrable. L'analyse du billard classique montre que l'on peut définir le flot jusqu'au rebond. Si la bille n'atteint pas une singularité, le rebond est bien défini. L'état après le rebond est bien une configuration admissible (puisqu'elle est dans un billard plus grand). Plus globalement on peut définir l'ensemble des trajectoires indéfiniment prolongeables ({\it régulières}) et s'assurer qu'elles forment un ensemble de mesure pleine. Des hypothèses géométriques simples permettent d'assurer que, même  si le temps entre les rebonds n'est pas uniformement  minoré, le nombre de rebonds par unité de temps reste uniformément borné. Ces résultats sont  formalisés dans la section~\ref{sec:definition}.


\subsection{Vitesse de fuite pour un domaine restreint}
Fixons un entier $K>0$. Nous allons  nous intéresser au casse-briques dans des domaines très particuliers : pour $\bfh\geq 0$,  
$$\Delta_{\bfh} = \{ \bfx \in \rr^2, \,  0\leq x_1 \leq K , \, x_2 \geq -\bfh \}. $$

Soit $\carre=[0,1]\times [0,1] \subset \rr^2$ le carré unité. On note $\carre_{\bf z} =  \carre + (z_1, z_2) $ les obstacles --- que nous appellerons ici {\em briques}.  On note $\cZ_K = \{0, 1, \ldots, K-1\} \times \zz_+$ et $\cC_K \subset \{0,1\}^{\cZ_K}$ l'ensemble des configurations contenant un nombre fini de $0$. On considère la famille d'obstacles $\cO = \{ \carre_{{\bf z}}, {\bf z} \in \cZ_K\}$.  Un élément $\eta  \in \cC_K$ décrit une famille d'obstacles $\cO(\eta) :=  \{ \carre_{{\bf z}}, {\bfz} \in \cZ_K,  \eta(\bfz)=1 \} \subset \cO$. 
\vip

Le casse-briques $\Phi := \Phi^{\bfh}$ est le casse-briques défini dans le domaine $\Delta^\cO_\bfh$ muni de la famille d'obstacles $\cO$. L'ensemble des positions initiales admissibles est naturellement  $\Delta_\bfh^{\cO}$, essentiellement réduit à une bande de largeur $K$ et de hauteur $\bfh$ qu'on munit de la mesure de Lebesgue.  Nous nous contenterons de considérer des positions initiales de la bille   dans l'intervalle $[0,K] \times \{-\bfh\}$ et donc d'angles dans $(0,\pi)$.

\begin{figure}
\includegraphics[width=60mm]{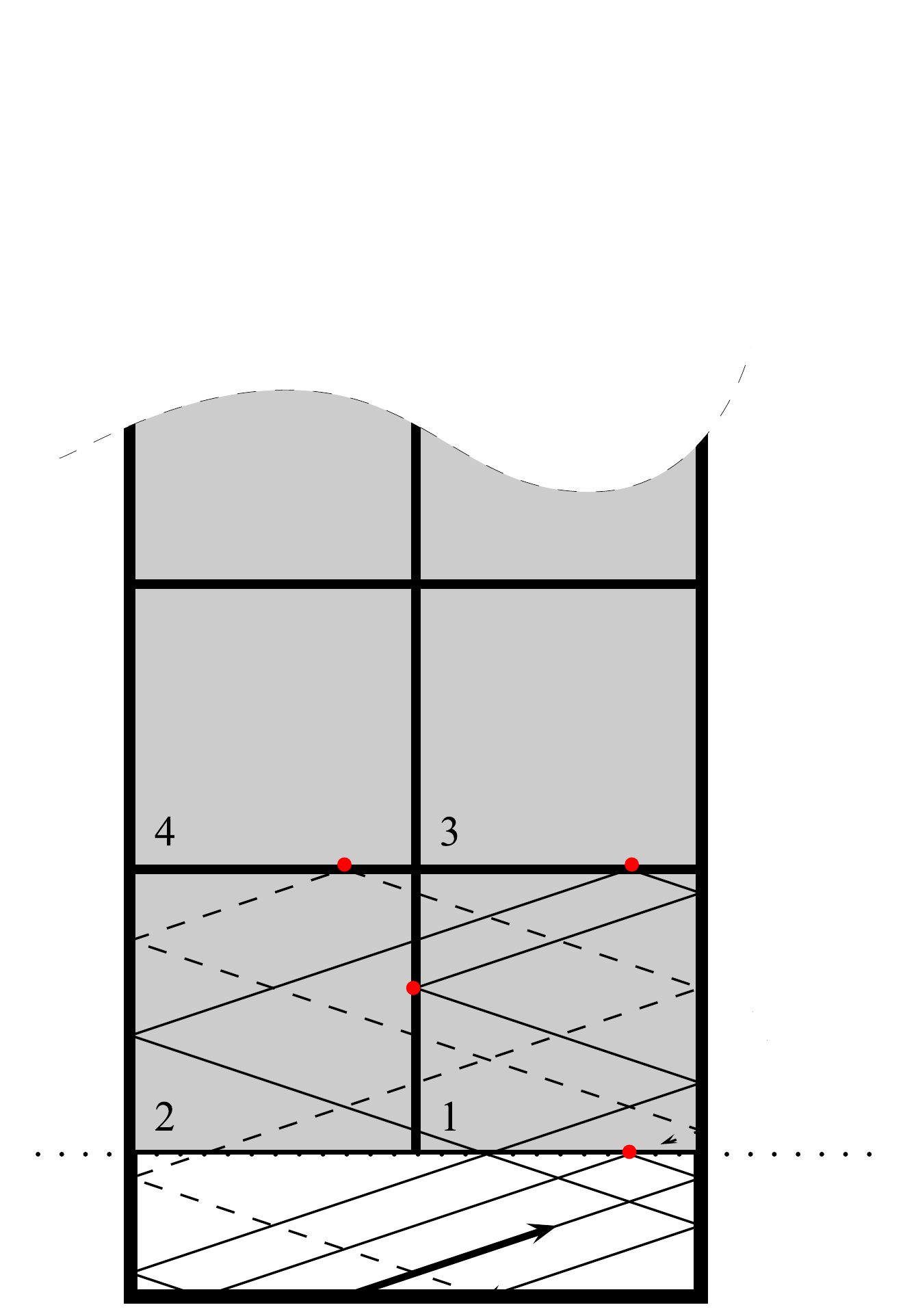}

  \caption{Début d'une trajectoire pour $K=2$, $\bfh = 1/2$ et $\tan \theta_0 = 1/3$. Les numéros dans les briques indiquent l'ordre dans lequel elles sont heurtées par la bille et donc disparaissent. Les deux premiers allers et retours sont en traits pleins ; le suivant en pointillés. }
  \label{fig:}
\end{figure}

\vip

Pour une configuration $\eta$, on définit la hauteur 
$$H(\eta) := \inf{\{ z_2 ; \bfz=(z_1,z_2)  \in \cZ_K, \eta(\bfz) =1 \}}$$ 
de la (ou des)  brique(s)  la(es)  plus basse(s) encore présente(s) dans la configuration $ \eta$ et on note $H(\bfx,\eta,\theta)  := H(\eta)$. 
Notre résultat le plus précis décrit la {\em vitesse de fuite} du casse-brique, c'est-à-dire le comportement asymptotique de $H(\Phi_t(\bfx_0, \eta_0, \theta_0))$, pour des valeurs très particulières des paramètres~:  
\begin{theorem}
\label{th:vitesse}
Fixons $K=2$. Soient $\theta_0$ avec $\tan{ \theta_0} =  1/4$, $\eta_0$ bien remplie et $\bfh< 1/10$. Alors, pour Lebesgue presque tout $\bfx_0 \in \Delta_\bfh^{\cO}$,  
\begin{equation}
\lim_{t \to \infty} \frac{H(\Phi^\bfh_t(\bfx_0, \eta_0, \theta_0))}{\sqrt{t}} = \sqrt{(1-4\bfh)/\sqrt{2}}
\end{equation}
\end{theorem}
Mais notre étude porte plus généralement sur la vitesse de fuite. Nous montrons d'abord, pour tout $K$ entier, que si $\theta_0$ est suffisamment petit, alors la brique la plus basse et le trou le plus haut restent essentiellement à la même hauteur. Dès lors, on montre que $1/K \leq H_t/\sqrt{t\sin{\theta_0}} \leq \sqrt{K}$. Mais surtout, une periodicité triviale permet alors de ramener le système à un système dynamique $\varphi$ sur un espace compact, et l'étude la vitesse de fuite à celle d'une somme ergodique au dessus de ce système.  La nature du système dynamique est particulièrement simple lorsque l'angle est rationnel : il s'agit d'une translation d'intervalle. Il est possible de l'expliciter complètement.  Elle est périodique pour des valeurs de $\theta_0$ et $\bfh$ rationnelles. Un calcul fini peut alors donner la vitesse de fuite en fonction de l'état initial. Lorsque $K=2$ et pour des valeurs de $\bfh$ irrationnelles suffisamment petites, nous menons l'étude jusqu'au bout  dans le cas où $\tan \theta_0 = 1/4$. Nous pensons que les idées développées peuvent  fonctionner pour d'autres paramètres, mais nous n'avons pas encore d'argument global.   Ces résultats sont démontrés dans la section~\ref{sec:restreint}.


\subsection{Régularité de certaines orbites pour le casse-briques dans le plan}
\begin{figure}
\includegraphics[width=60mm]{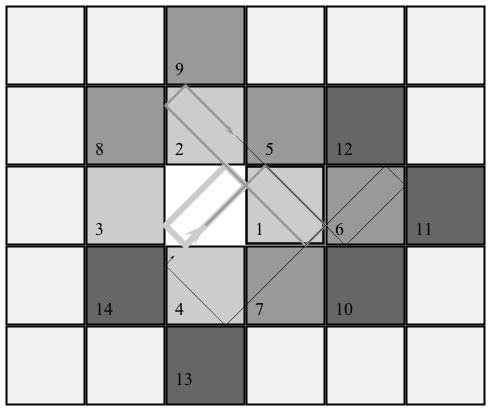}

  \caption{Exemple de trajectoire : l'angle initial satisfait $\tan \theta_0 = 1$ ; les numéros dans les briques indiquent l'ordre dans lequel elles sont heurtées par la bille (et donc disparaissent) ; le niveau de gris accentue cette information. La trajectoire est représentée par une ligne de plus en plus fine pour aider à distinguer les passages multiples. }
  \label{fig:zdeux1}
\end{figure}

\begin{figure}
\includegraphics[width=70mm]{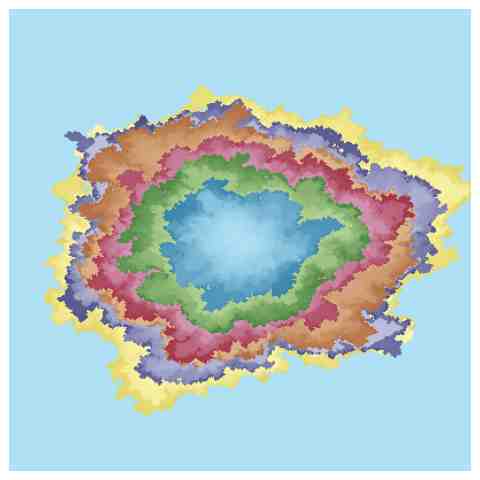}
\includegraphics[width=70mm]{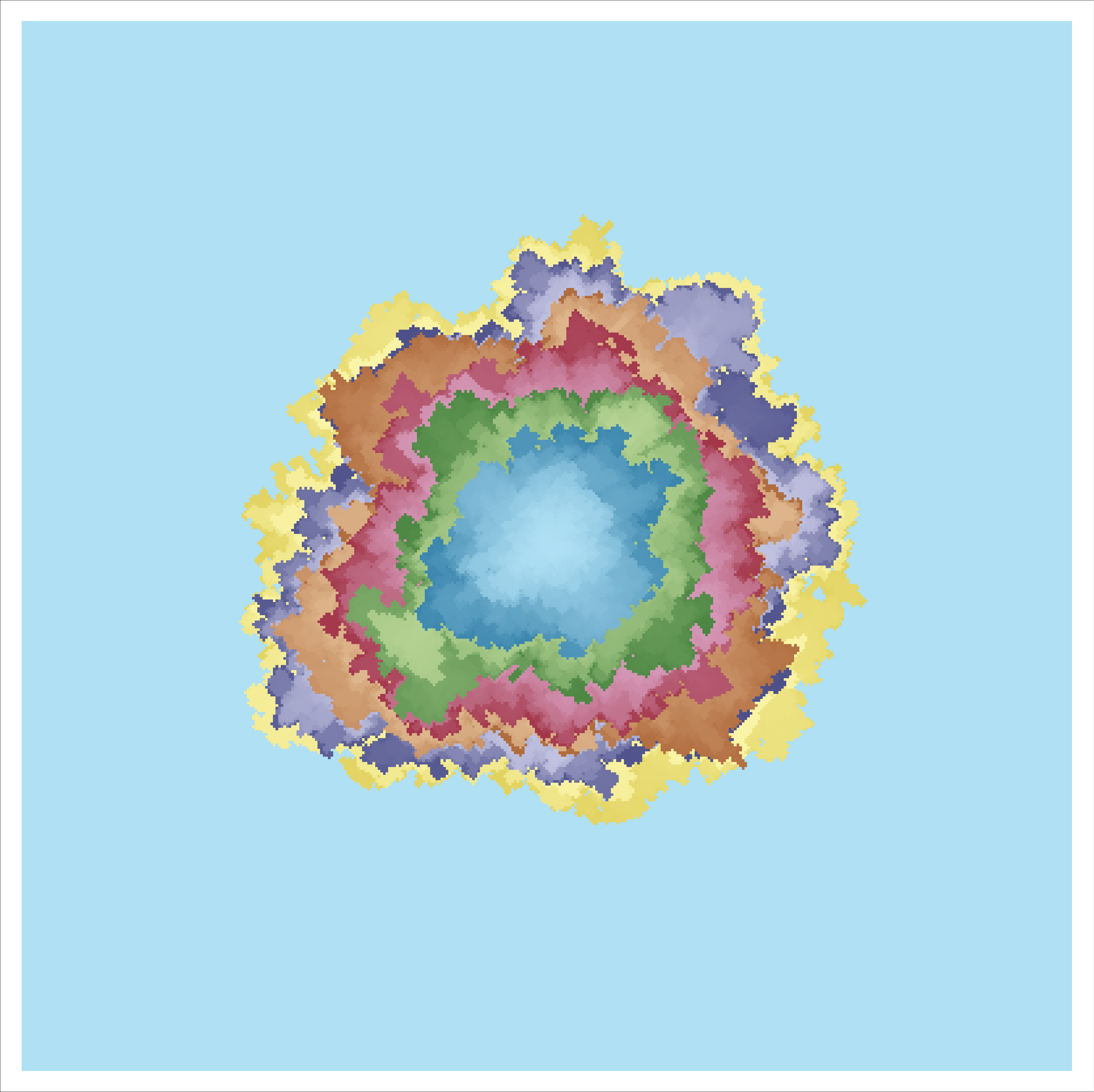}
\includegraphics[width=70mm]{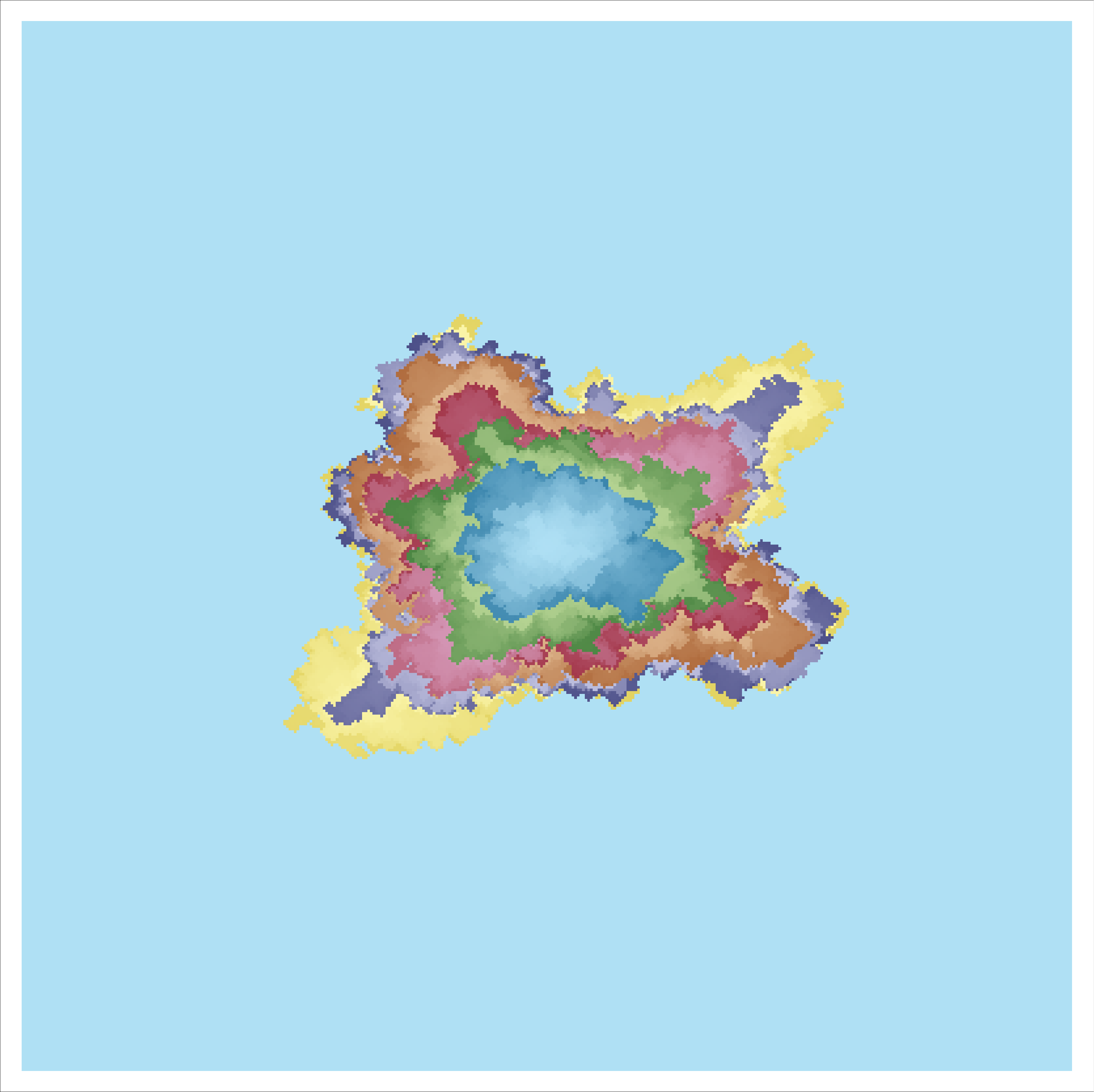}
\includegraphics[width=70mm]{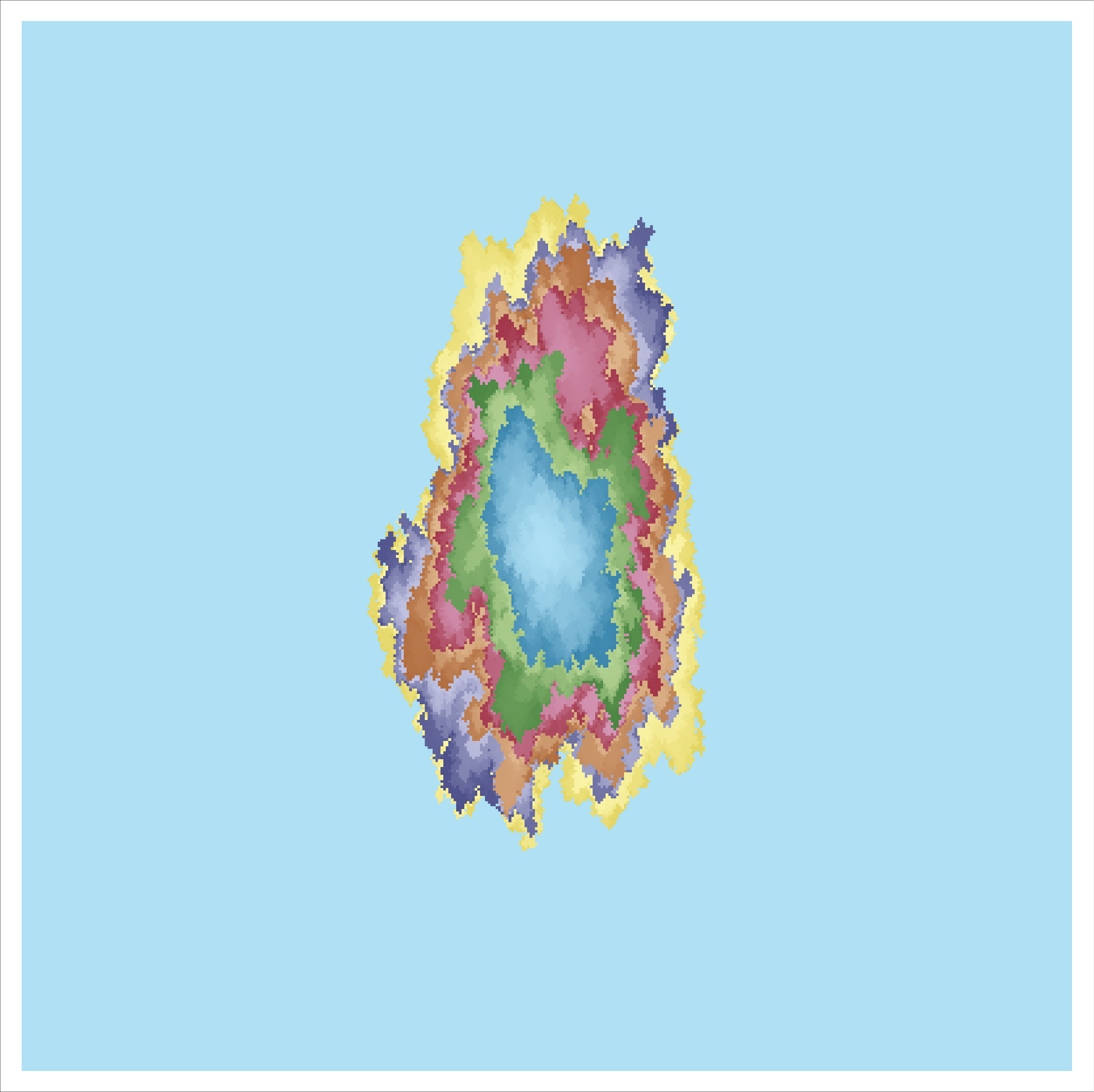}

  \caption{Ces figures sont obtenues en colorant les carrés de $\zz^2$ (ici des pixels) d'une couleur suivant une échelle donnée (bleu, vert, rouge, marron, violet, jaune) dans l'ordre ou les obstacles sont heurtés par la bille pour des trajectoires de longueur de l'ordre d'un million et d'angles respectifs : $\tan \theta_0 = 3/2$,  $\tan \theta_0 = 21/22$, $\theta_0 = 4 \pi /5$ et $\theta_0=3 \pi / 8$. }
  \label{fig:zdeux2}
\end{figure}

Notre première idée était d'étudier le casse-briques dans le plan rempli d'obstacles carrés : $\Delta = \rr^2$ et $\cO = \{ \carre_{{\bf z}}, {\bf z} \in \zz^2 \setminus \{(0,0) \} \}$. Nous nous sommes heurtés à des difficultés combinatoires inattendues : nous n'avons pu comprendre l'asymptotique de la dynamique que pour un groupe très particulier d'angles d'incidence : $\tan \theta = 0, 1, 2, 3$ ; le comportement du système pour les autres angles reste parfaitement mystérieux, comme l'illustrent les figures~\ref{fig:zdeux1} et \ref{fig:zdeux2}. Nous proposons,  dans la section~\ref{sec:zdeux}, un cadre permettant de décrire les régularités observées sur ces exemples, introduisant notamment l'idée d'orbites relativement périodiques ; mais les  résultats sont très partiels : nous apportons plus de questions que de réponses,  espérant que d'autres sauront aller plus loin. 

\section{Définition et existence du casse-brique}
\label{sec:definition}
\subsection{Isométries par morceaux}
Commençons par rappeler quelques définitions usuelles de systèmes dynamiques : une rotation d'angle $2\pi \alpha$ (du cercle unité) peut-être vue aussi comme une translation du tore $\mathbb{T}^1$ et vue sur l'intervalle $[0,1]$ (en identifiant $0$ et $1$) comme l'application $x \mapsto x + \alpha$ (modulo $1$). C'est alors un exemple d'échange d'intervalles (échange de deux intervalles), une classe plus vaste d'applications : soit $(I_1, \ldots, I_d)$ une partition de $[0, 1]$ en $d$ intervalles consécutifs de longueurs respectives $(\lambda_1, \ldots, \lambda_d)$. Si on choisit une famille $(t_1, \ldots, t_d)$ de réels entre $0$ et $1$, on peut définir une {\em translation d'intervalles} en considérant l'application translatant, pour chaque $i \in \{ 1, \ldots, d\}$  les points  de $I_i$ d'un vecteur $t_i$ (modulo 1). Si une telle application est bijective, on dit que c'est un {\em échange d'intervalles}. Il existe alors 
une permutation $\Pi$ sur $\{1, \ldots, d\}$ et une partition une partition  $(J_1, \ldots J_d)$ de $[0,1]$ en intervalles consécutifs  de longueurs respectives $(\lambda_{\Pi(1)}, \ldots, \lambda_{\Pi(d)})$ telles que l'échange d'intervalles  envoie l'intervalle $I_i$ sur l'intervalle $J_{\Pi^{-1}(i)}$ (qui est de même longueur). 

Etant donnée une application $\varphi$  (par exemple de l'intervalle dans lui même) et une partie $A$ de son ensemble de définition, on appelle application induite de $\varphi$ sur $A$ l'application définie sur $A$ par $\varphi_A(x) = \varphi^{\tau(x)}(x)$ où $\tau(x) = \inf{\{n>0 : \varphi^n(x) \in A \} }$ est le premier temps de retour de $x$ dans $A$. 

\vip

Les propriétés dynamiques de la rotation d'angle $2\pi \alpha$ dépendent de la rationalité de $\alpha$~: elle est périodique si et seulement si $\alpha$ est rationnel. Autrement toutes  ses orbites sont denses et, en fait celles-ci satisfont une propriété d'équipartition sur l'intervalle : ce système dynamique admet une unique mesure invariante (unique ergodicité), la mesure de Lebesgue, qui est donc ergodique ; cela assure un comportement raisonnable des moyennes ergodiques presque partout~; en fait, pour ces systèmes, la situation est encore plus favorable : le théorème ergodique est valable en tout point : pour un intervalle $A$, $1/n \sum_{k=0}^{n-1}  1_A(\varphi^k(x)) \to |A|$ uniformément. De tels résultats existent aussi pour les échanges d'intervalles sous certaines hypothèses. Mais les propriétés ergodiques sont plus délicates ; en particulier, l'ergodicité de la mesure de Lebesgue n'est pas garantie en général. Les translations d'intervalles ne sont pas bijectives~; pour se ramener à des bijections, on peut se placer sur leur ensemble limite ; mais plusieurs situations peuvent se produire ; l'ensemble limite peut être une réunion d'intervalles, mais aussi un ensemble de Cantor.  On pourra voir \cite{bos1} pour plus de détails sur ces systèmes. 


\subsection{Billards polygonaux}
On définit de manière classique le flot $(\bill^{\Delta}_t)_{t \in \rr}$ du billard dans un domaine polygonal $\Delta$. Il s'agit d'un billard polygonal, éventuellement non borné.  C'est une famille d'applications du fibré unitaire tangent  $\bill^{\Delta}_t :T\Delta = \Delta \times \cercle \to T \Delta$ formant un groupe. 

Soit $x \in \Delta$ et $v \in T_x\Delta$ tels quel  $\tau = \inf\{s>0 ; x+sv \in \partial \Delta\} >0$. Notons $x_\tau = x+\tau v \in \partial \Delta$, considérons le segment bordant $\Delta$ en $x_\tau$ et sa normale interieure $\bf n$ relevée dans $T_{x_\tau}\Delta$ (cela n'a de sens que si $x_\tau$ n'est pas une singularité du bord). Notons $v'$ le symétrique de $v$ par rapport à $\bf n$. On definit alors  $\bill^{\Delta}_t(x,v) = (x+tv,v)$ pour tout $t<\tau$ et $\bill^{\Delta}_\tau (x,v) = (x+\tau v, v')$. 

Pour que cette définition fasse sens il convient de se restreindre, pour les points de la frontière, à ceux dont le vecteur vitesse pointe vers l'interieur de $\Delta$. Enfin, il faut exclure les points  dont l'orbite rencontrerait une singularité du bord du domaine ; il est  possible  de montrer que ceux-ci forment un ensemble de mesure de Lebesgue  nulle. Ainsi le billard est défini sur un ensemble invariant de mesure pleine ; les trajectoires sont alors continues à droite et limitées à gauche. La famille d'applications $(\bill^{\Delta}_t)_{t\in \rr}$ définies sur cet ensemble forme un groupe : 
$\bill^{\Delta}_{t+s}= \bill^{\Delta}_t \circ \bill^{\Delta}_s$.  Observons enfin que lorsque les angles entre les cotés du billard sont rationnels, pour chaque angle initial, l'application de premier retour du flot sur les frontières du domaine est une isométrie par morceaux. Pour plus de détails et la preuve de résultats classiques que nous utiliserons de manière plus ou moins explicite, on pourra se référer à \cite{tab1}.

\subsection{Définition du casse-briques}

On se donne un domaine $\Delta$ (pas nécessairement compact, possiblement $\rr^2$ tout entier) et une famille dénombrable d'obstacles $\cO= \{B_z, z \in \cI\}$ d'interieurs disjoints. On supposera (pour simplifier) que toute boule n'intersecte qu'un nombre fini d'obstacles. On considère pour une famille d'obstacles $\cO$ l'ensemble $\Delta^\cO$ constitué du domaine $\Delta$ privé des obstacles :   $\Delta^\cO = \overline{\Delta \setminus \cup_{B \in \cO} B}$.

On appelle {\em configuration} notée $\eta$ un élément de $\{0,1\}^\cI$. Nous nous limiterons à l'ensemble des  configurations, dites {\em bien remplies}, 
$$\cC = \{\eta \in \{0,1\}^\cI ; \sum_{z \in \cI} (1-\eta(z)) < \infty\}.$$

A une configuration $\eta$ on associe  une partie de $\cO$ en posant $\cO(\eta)  = \{B_z ; z \in \cI, \eta(z) =1\}$  plus petite que $\cO$  mais ne différant de $\cO$ que par un nombre fini d'éléments (nous appelerons parfois abusivement cette liste aussi configuration). On munit les ensembles  $\Delta^\eta := \Delta^{\cO(\eta)}$ de la mesure de Lebesgue. Etant donné un obstacle $P \in \cO(\eta)$, on note $\eta^{P}$ la configuration associée à la famille d'obstacles $\cO(\eta)$ privée de l'obstacle $P$~:  $\cO(\eta^{P}) = \cO(\eta) \setminus\{P\}$. 

\vip

Un état du système est décrit par un triplet composé d'une {\em position} $\bfx \in \Delta$, d'une {\em configuration} $\eta \in \cC$  et d'un {\it angle} $ \theta \in  \cercle \simeq [0,2\pi)$ caractérisant un vecteur {\em vitesse}  $v(\theta) = (\cos \theta, \sin \theta)$ unitaire. Un état $(\bfx,\eta,\theta)$ du système est dit {\em admissible} si $\bfx$ est dans  $\Delta^\eta$, et, pour les points de la frontière, si le vecteur vitesse pointe vers l'intérieur du domaine (strictement), c'est-à-dire si il existe $\epsilon >0$ tel que pour tous $0 < s < \epsilon$, $\bfx + s v(\theta)$ est dans l'interieur de $\Delta^\eta$. On peut interpréter le couple $(\bfx, v(\theta))$,  pour chaque configuration $\eta$, comme un élément du fibré unitaire tangent $T \Delta^\eta$ de $\Delta^\eta$. Plus formellement,  
on définit, pour $\eta \in \cC$ et $\bfx \in \Delta^\eta$, 
$$T^+_\bfx \Delta^\eta = \{ v \in \cercle : \exists \epsilon >0, \, \forall 0 < s < \epsilon, \, \bfx + s v \in \hbox{int}(\Delta^\eta) \}  \subset T_\bfx \Delta^\eta. $$
Cela nous permet de définir l'ensemble des états admissibles comme : 
$${\cA } = \left\{ (\bfx, \eta, \theta) :  \eta \in \cC, \bfx \in \Delta^\eta,  v(\theta) \in T^+_\bfx \Delta^\eta) \right\}. $$
Pour chaque configuration $\eta$, l'espace fibré unitaire tangent  $T \Delta^\eta$ est naturellement muni de la mesure de Lebesgue. Ainsi, par restriction, on a une version de la mesure de Lebesgue sur chaque coordonnée de $\cA$. Cela fournit une mesure sur la réunion dénombrable (sur l'ensemble des configurations bien remplies) que constitue $\cA$ et que nous noterons $\lambda$. 

\vip

Etant donnés une configuration $\eta$,  un point $\bfx$ et un angle $\theta$  tels que $(\bfx,v(\theta)) \in T^+\Delta^\eta$, la trajectoire $(\bill^\eta_t(\bfx,v(\theta)))_{t \geq 0}$ du billard dans $\Delta^\eta$ est définie tant qu'elle ne rencontre pas une singularité, i.e. un point du bord $\partial\Delta^\eta$ où la normale n'est pas définie. Il peut arriver un moment où elle rencontre un obstacle : on note $\tau$  le premier instant où la bille rencontre un obstacle et $\beta$ l'obstacle rencontré,  
$$\tau =  \tau(\bfx,\eta,\theta) = \inf\{ t \geq 0 \, : \, \bill^{\eta}_t(\bfx,v(\theta)) \in  \cup_{z \in \cI,  \eta(z)=1} B_z \times \cercle\},$$
et
$$ \beta(\bfx,\eta,\theta) = B_{\zeta}, \hbox{ où } \zeta \hbox{ est l'unique indice } \zeta \in \cI \hbox{ tel que } \bill^\eta_\tau(\bfx,v(\theta)) \in B_\zeta  \times \mathbb{S}_1.$$ 
On définit, lorsque c'est possible, un semi-flot $(\Phi_t)_{t \in \rr}$ sur $\cA$ en posant 
$$\Phi_t(\bfx, \theta, \eta) = \left\{ \begin{array}{ll} (\bill^\eta_t(\bfx,v(\theta)),\eta) & \hbox{ si } 0 \leq t <  \tau  \\  (\bill^\eta_{\tau}(\bfx,v(\theta)),\eta^{\beta(\bfx,\eta,\theta)}) & \hbox{ si } t=\tau.\end{array} \right. $$
Cela n'a de sens que si $\beta(\bfx,\eta,\theta)$ est uniquement défini et donc si $\bill^\eta_\tau(\bfx,v(\theta))$ n'appartient qu'à un seul obstacle, et si le rebond est bien défini pour le billard, c'est-à-dire si $\bill^\eta_\tau(\bfx,v(\theta))$ n'est pas une singularité du bord de  $\Delta^\eta$.

\subsection{Trajectoires infinies}
Nous allons voir que l'ensemble des états initiaux pour lesquels la trajectoire peut être prolongée indéfiniment est de mesure de Lebesgue pleine. 
La définition des  configurations admissibles garantit que le début de la  trajectoire est bien défini si on part d'un état admissible puisque c'est vrai pour le billard : 
\begin{lemma}[Existence locale]
\label{existenceloc}
Soit $(\bfx, \eta, \theta) \in \cA$. Il existe $t_0 >0$ tel que pour tout $0\leq t \leq  t_0$, $\Phi_t(\bfx, \eta, \theta)$ soit bien défini. 
\end{lemma}
On peut ainsi prolonger la trajectoire jusqu'à ce qu'on rencontre un bord du domaine $\Delta^\eta$. La seule question délicate est celle de savoir si on peut prolonger les trajectoires lorsqu'elles rencontrent le bord du billard. Si la trajectoire au delà de cet instant est bien définie pour le billard, on dit  qu'il y a eu un {\em rebond}. Dans le billard, on ne peut prolonger la trajectoire que si le point du bord atteint admet une normale. Ici, on doit aussi lever une éventuelle ambiguité : il faut que $\beta(\bfx,\eta,\theta)$ soit bien défini. On va donc demander aussi que le point de rebond n'appartienne qu'à un seul obstacle. On peut être amenés à distinguer les rebonds sur les bords du domaine $\Delta$ (qui ne modifient pas la configuration) et les rebonds sur les bords des obstacles. 

On définit pour chaque configuration $\eta$ l'ensemble $S^\eta$ des sommets de $\Delta$ et des polygones consitutant $\cO(\eta)$. Observons que  tous les autres points de la frontière de $\Delta^\eta$ admettent une normale et appartiennent à un unique obstacle ou au bord de $\Delta$ ; en effet, les points appartenant à plusieurs obstacles de $\cO(\eta)$ qui ne sont sommets d'aucun de ces obstacles ne peuvent appartenir à la frontière de $\Delta^\eta$.  Notons que $S^\eta \subset S^\cO$. 
\begin{lemma}[Prolongation]
\label{prolongation}
Soit $\bfx_\tau$  la position de $\Phi_\tau^\eta(\bfx,v(\theta))$ au premier rebond. Si $\bfx_\tau$  n'appartient pas à $S^\eta$, alors $\Phi_\tau(\bfx, \eta, \theta) \in \cA$ et on peut prolonger le semi-flot au delà de $\tau$.  
\end{lemma}
\begin{proof}
Il n'y a qu'un nombre fini d'obstacles dans une boule centrée en $\bfx_\tau$. Cela entraine qu'il existe une boule autour de $\bfx_\tau$ dans laquelle il n'y a pas d'obstacle de $\cO(\eta)$  (autre que $\beta(\bfx,\eta,\theta)$), sinon $\bfx_\tau$ appartiendrait à un sommet de cet obstacle (puisqu'il n'est pas sur un sommet de $\beta(\bfx,\eta,\theta)$ et que  tout voisinage de $\bfx_\tau$ intersecte l'interieur de $\Delta^\eta$). En particulier on peut prolonger la trajectoire $\Phi_t^\eta(\bfx,v(\theta))$ au delà de $\tau$ dans $\Delta^\eta$ et donc à fortiori dans $\Delta^{\eta^{\beta(\bfx,\eta,\theta)}}$. 
\end{proof}

\vip

On note $\eta_{t-}$ la configuration juste avant l'instant $t$ et on considère  l'ensemble 
$$ {\cS}  = \{ (\bfx, \eta, \theta)  \in \cA \, : \, \exists t>0, \, \Phi_t (\bfx,\eta,\theta) \in S^{\eta_{t-}}\}.$$
Le lemme~\ref{prolongation} garantit que les trajectoires des états de $\cA^* = \cA \setminus \cS$ peuvent être prolongées indéfiniment ; on dira qu'elles sont {\em régulières}. 
L'objet de la suite est de montrer que 
\begin{proposition}
La mesure de Lebesgue de $\cS$ est nulle : $\lambda(\cS) =0$. 
\end{proposition}
\begin{proof}
On considère  les ensembles 
$$\cS (\eta,\theta) = \{ \bfx \in \Delta^\eta \, : \, \exists t>0, \, \Phi_t (\bfx,\eta,\theta) \in S^{\eta_{t-}}\}.$$
Nous allons montrer que, pour chaque configuration $\eta$ et chaque angle $\theta$, $\cS(\eta, \theta)$ est une réunion dénombrable de segments de $\Delta^\eta$.  

\vip

Soit $N$ un entier. Appelons $S_N^\eta(\theta)$ l'ensemble  des points de $\Delta^\eta$ qui avec une vitesse $v(\theta)$ atteignent une singularité (un point de $S^\cO$) après moins de $N$ rebonds.  
Pour $N=1$ et pour toute configuration $\eta$, $S_1^\eta(\theta)$ est une réunion dénombrable de segments de pente $\theta$.

Soit $N\geq 1$. Supposons que, pour toute configuration $\eta$ et tout angle $\theta$,  l'ensemble $S_N^\eta(\theta)$ des points de $\Delta^\eta$ qui avec une vitesse $v(\theta)$ atteignent une singularité (ou un ensemble dénombrable contenant les singularités) après moins de $N$ rebonds soit une réunion dénombrable de segments de pente $\theta$. Nous allons montrer que cette assertion reste vraie au rang $N+1$.

Fixons une configuration $\eta$ et un angle $\theta$. L'ensemble des points qui atteignent $S^\cO$ en $N+1$ rebonds se décompose en, d'une part, ceux qui atteignent $S^\cO$ en un  coup, et, d'autre part, ceux qui font un rebond (donc n'atteignent pas $S^\eta \subset S^\cO$ au premier rebond), puis, atteignent $S^\cO$ en $N$ coups. Si on note $\tau$ l'instant du premier rebond, $\eta_\tau$ et $\theta_\tau$ la configuration et l'angle après le rebond,  cela s'écrit : 
$$S_{N+1}^\eta(\theta)= S_1^\eta(\theta) \cup \{ \bfx \in \Delta^\eta \setminus S_1^\eta(\theta) : \Phi_\tau(\bfx, \eta, \theta) \in S_N^{\eta_\tau}(\theta_\tau)\}. $$
Le premier ensemble ne pose pas de problème ; analysons le second. Le segment $s_\tau \subset \partial \Delta^\eta \subset \Delta^{\eta_\tau}$ sur lequel rebondit la trajectoire est bien défini et est atteint transversalement sinon on serait passé sur une singularité.   L'angle qu'il fait avec l'horizontale ne peut pas  être égal à l'angle d'incidence $\theta$, et donc pas non plus à l'angle  $\theta_\tau$ de la trajectoire après le rebond.   
Mais $S^{\eta_\tau}_N(\theta_\tau)$ est une réunion dé,nombrable de segments formant un angle $\theta_\tau$ avec l'horizontale~; ainsi son intersection $s_\tau \cap S_N^{\eta_\tau}(\theta_\tau)$  avec le segment $s_\tau$ est un ensemble dénombrable et l'ensemble des points de $\Delta^\eta  \setminus S_1^\eta(\theta)$ qui atteignent $s_\tau \cap S_N^{\eta_\tau}(\theta_\tau)$ est une réunion dénombralbe de segments d'angle $\theta$. Ainsi  $\{ \bfx \in \Delta^\eta \setminus S_1^\eta(\theta) : \Phi_\tau(\bfx, \eta, \theta) \in S_N^{\eta_\tau}(\theta_\tau)\}$ est la réunion (dénombrable) sur l'ensemble des segments $s$ formant le bord de $\Delta^\eta$ de ces réunions dénombrables de segments de pente $\theta$. On conclut que  $S_{N+1}^\eta(\theta)$ aussi. 

On a ainsi montré par récurrence que pour toute configuration $\eta$, tout angle $\theta$ et  tout entier $N$,  $S_N^\eta(\theta)$ est une réunion dénombrable de segments de pente $\theta$. En considérant la réunion de ces ensembles sur tous les entiers, on déduit l'assertion annoncée. 

\vip

On conclut la preuve en rappelant que l'ensemble des configurations bien remplies est dénombrable et en intégrant sur les angles (Fubini) un ensemble de mesure nulle.  
\end{proof}
\begin{remark}
Observons que la preuve entraîne en particulier que, pour chaque angle $\theta$, l'ensemble $\partial \cS(\theta) = \bigcup_{\eta  \in \cC}\cS(\eta,\theta)\cap  \partial \Delta^{\eta}$ est  dénombrable. 
\end{remark}

Ainsi, si on se restreint à l'ensemble de mesure pleine $\cA^*$, le casse-briques définit un semi-flot~:  pour tous $t,s \geq0$, $$\Phi_{t+s} = \Phi_t \circ \Phi_s.$$ 
Etant donné un état initial $(\bfx_0,\eta_0,\theta_0) \in \cA^*$, on notera~: 
$(\bfx_t,\eta_t,\theta_t) = \Phi_t(\bfx_0,\eta_0,\theta_0).$




\subsection{Invariance par translation.} Une translation de l'ensemble de la figure ne change  pas la dynamique. On peut formaliser cette invariance par translation en posant pour tout vecteur $\bf{u} \in \rr^2$,  $\cO(\eta+{\bf u}) :=  \{ B_z+{\bf u}, z \in \cI, \eta(z)=1\} \subset \cO + {\bf u} :=  \{ B_z+{\bf u}, z \in \{0,1\}^\cI \}$ et $(\bfx,\eta,\theta)+\bf{u} := (\bfx +\bf{u} ,\eta+\bf{u},\theta)$, et en notant $\Phi_t^{(\Delta+ \bf{u})}$ le casse-briques dans le domaine $\Delta+{\bf u}$ muni des obstacles $\cO +{\bf u}$ ;    avec ces notations, l'identité suivante, valable pour tout $t>0$, découle directement de la définition géométrique du système~:
\begin{equation}
\label{translationgenerale}
\Phi_t^{(\Delta+ \bf{u})}(\bfx+{\bf u} , \eta + {\bf u}, \theta) = \Phi_t^{(\Delta)}(\bfx, \eta, \theta)+\bf{u}. 
\end{equation}

\vip

Terminons par une série de remarques : 
\begin{itemize}
\item Si on autorisait toutes les configurations possibles, on aurait l'avantage d'avoir un ensemble compact. Mais on préfère s'intéresser uniquement à celles dans lesquelles on a oté un nombre fini de briques, les configurations bien remplies, seules accessibles en un temps fini. Cet ensemble de configurations est par définition stable (du fait qu'on ait qu'un nombre fini de rebonds par unité de temps). Observons que la propriété de connexité de $\Delta^\eta$ est elle aussi stable par le flot. 
\item Le nombre de briques ne peut que décroître. On peut s'interroger sur la limite (combinatoire) du système. Soit la configuration ne change plus à partir d'un certain rang (et on termine sur un billard classique) soit le système est transitoire (et l'éventuelle limite n'est pas une configuration bien remplie). 
\item Insistons de nouveau sur le fait que le flot n'est qu'un semi-flot. Il n'est pas injectif. Comme pour les billards classiques, on peut essayer, spécialement dans le cas d'angles "rationnels" de lui associer un flot sur une surface. Nos tentatives dans ce sens amènent à construire en fait un semi-flot sur une surface branchée non compacte. Voir la remarque~\ref{rem:geom}. 
\end{itemize}

\section{Domaine restreint}
\label{sec:restreint}
Fixons  un entier $K$ et revenons au casse-briques  $\Phi := \Phi^{\bfh}$ défini dans le domaine $\Delta_\bfh$ pour $\bfh \geq 0$ muni de la famille d'obstacles $\cO = \{ \carre_{{\bf z}}, {\bf z} \in \cZ_K\}$ qu'on appelera {\em briques}. L'ensemble des positions admissibles est naturellement  $\Delta_\bfh^{\cO}$, essentiellement réduit à une bande de largeur $K$ et de hauteur $\bfh$.  Nous nous contenterons de considérer des positions initiales de la bille  dans l'intervalle $[0,K] \times \{-\bfh\}$.   Observons que l'admissibilité entraîne alors que l'angle appartient à $(0,\pi)$. 
\vip

Nous commençons par faire quelques remarques élémentaires sur le comportement du casse-briques dans un tel domaine : 
\begin{itemize}
\item Rappelons que l'ensemble des configurations bien remplies $\cC_K$ est stable puisque le nombre de briques heurtées par unité de temps est fini.  
\item Le semi-flot est défini localement sur l'ensemble des états admissibles. Le semi-flot  est bien défini globalement (pour tout $t>0$) sur l'ensemble du domaine auquel on enlève, pour chaque angle, un nombre dénombrable de segments. C'est l'ensemble des états {\em réguliers}.  Cela permet en particulier de  voir que, pour tout angle $\theta$ fixé,   l'ensemble des positions initiales admissibles sur la frontière  qui rencontrent une singularité est dénombrable. En particulier,  si $\theta \in (0,\pi)$, l'ensemble des positions initiales d'ordonnée $-\bfh$  qui rencontrent une singularité est dénombrable.  
\item Partant d'un état initial $(\bfx_0,\eta_0,\theta_0)$ régulier, on observe que l'angle reste dans l'ensemble $\{\theta_0, \pi-\theta_0,-\theta_0,\pi+\theta_0\}$ ; on note $\tilde \theta_0$ l'unique élément de cet ensemble qui est dans $[0,\frac{\pi}{2}]$. Cela tient au fait que les parois du billard comme celles des briques sont soit horizontales, soit verticales. Ainsi, 
$$\forall (\bfx_0,\eta_0,\theta_0) \in \cA^*,  \theta_t \in \{\theta_0, \pi-\theta_0,-\theta_0,\pi+\theta_0\}.$$
\end{itemize}
\subsection{Décomposition des trajectoires en allers-retours}
On va decrire les trajectoires en termes de retours à la base de $\Delta_K$. On peut distinguer les parties montantes ($\theta_t \in [0,\pi]$) et les parties descendantes ($\theta_t \in [\pi, 2\pi]$). Pendant ces périodes, la vitesse verticale de la bille est constante, égale à $\pm \sin{\tilde \theta_0}$. Loin des briques, sur l'horizontale, la bille fait des allers-retours entre les deux parois fixes de $\Delta_K$. Lorsqu'elle arrive dans la zone où il y a des briques, les choses se compliquent, mais on peut espérer que la bille finisse par revenir dans l'autre sens après avoir heurté une paroi horizontale et revienne à la base.


\subsubsection{Retours à la base}
On appelle {\em base} de $\Delta_\bfh$ et on note  $B_\bfh$, l'ensemble des états admissibles $((x_1,x_2),\eta,\theta)$ satisfaisant $x_2=-\bfh$ et $\theta \in (0,\pi)$, c'est-à-dire que $B_\bfh=([0,K] \times \{-\bfh\}) \times \conf_K \times(0,\pi)$. On note  $B_\bfh^*= B_\bfh \cap \cA^*$ son intersection avec les  états réguliers.  
Pour un angle fixé $\theta$, on note $B_{\bfh,\theta} =([0,K] \times \{-\bfh\}) \times \conf_K \times \{\theta, \pi-\theta\}$ et $B^*_{\bfh,\theta}$ son intersection avec $B_\bfh^*$. Observons que c'est une réunion  d'intervalles privée d'un ensemble dénombrable de points.  

\vip
On note, pour $(\bfx,\eta,\theta)  \in B^*_\bfh$, $T(\bfx,\eta,\theta) = \inf \{t >0 \, : \, \Phi_t(\bfx,\eta,\theta) \in B^*_\bfh\}$ et on  définit une application $\phi_\bfh : B_\bfh^*  \cap \{T<\infty\} \to B_\bfh^*$ en posant 
$$\phi_\bfh(\bfx,\eta,\theta) = \Phi^\bfh_{T(\bfx,\eta,\theta)}(\bfx,\eta,\theta).$$ 
Observons que $\phi_\bfh$ envoie $B_{\bfh,\theta}^*  \cap \{T<\infty\}$ dans $B_{\bfh,\theta}^*$.  Le flot $(\Phi^\bfh_t)_{t>0}$ est une suspension au dessus de $\phi_\bfh$ dont $T$ est la fonction "toit". Lorsque cela est possible, on définit $T_n= T \circ \phi_\bfh^n$ le $n$ième temps de retour. Le temps total écoulé au $n$ième retour est noté $\tau_n= \sum_{k=0}^{n-1} T_k$. On appelle {\em aller-retour} la trajectoire du flot entre deux retours à la base. 

\begin{remark} Il n'est pas évident que $T < \infty$ même si on part d'une configuration bien remplie.   Il semble  que cela soit faux comme le laisse présumer l'exemple construit avec $\tan{\theta_0} = 3$ dans le casse-briques sur $\rr^2$ entier (voir Section~\ref{sec:exemples}). En revanche nous conjecturons que si la tangente de l'angle $\theta$ est irrationnelle, alors pour toute configuration $\eta$ bien remplie et toute position $\bfx$ telles que $(\bfx, \eta, \theta) \in B_{\bfh,\theta}^*$,   on a bien $T(\bfx, \eta, \theta) <\infty$. 
\end{remark}

\subsubsection{Hauteur}
Pour une configuration $\eta$, on définit la hauteur $H = H(\bfx,\eta,\theta)  = H (\eta)$ 
de la brique la plus basse et la hauteur $H ^+$ de la case juste au dessus de la case  vide la plus haute~: 
$$H(\eta) := \inf{\{ z_2 \in \zz ; \bfz=(z_1,z_2)  \in Z_K, \eta(\bfz) =1 \}}$$ 
et 
$$ H^+(\eta) := \sup{\{ z_2 \in \zz ; \bfz=(z_1,z_2)  \in Z_K, \eta(\bfz) =0 \}} + 1.$$ 
 Etant donné un état initial $(\bfx_0,\eta_0,\theta_0)$, on définit $(\bfx_t,\eta_t,\theta_t) = \Phi^\bfh_t(\bfx_0,\eta_0,\theta_0)$, puis $H _t$ (et $H ^+_t$) comme $H _t = H (\eta_t)$.   De la même manière, on définit 
$(\bfx_n,\eta_n,\theta_n) = \phi_\bfh^n(\bfx_0,\eta_0,\theta_0)$ et $H _n = H (\eta_n)$. 
On s'intéresse à la vitesse de fuite, i.e. au comportement asymptotique de $H _t$ ou/et de $H _n$. 

\begin{remark}
[Cas triviaux] Le cas où $\theta_0 = 0$ est particulier. Partant de la base il ne correspond pas à un état admissible. Mais de toutes façons, la dynamique ne serait pas intéressante (orbite horizontale periodique de periode $2K$). Le cas $\theta_0 = \pi/2$ est simple : si l'on part d'une position d'abscisse $x$ non entière, la bille effectue des allers-retours entre la base et une brique située de plus en plus haut. Plus précisemment, notons  $n_k = 2 \sum_{\ell=0}^{k-1}  (\bfh +\ell) = 2 k \bfh +  k (k+1)$ et $\eta_{k,z} = 1 - \sum_{i=1}^k 1_{\{i,z\}}$. Un calcul direct montre que $\Phi^{\bfh}_{n_k} ((x,-\bfh), \eta_0, \pi/2) = ((x,-\bfh), \eta_{k,\lfloor x \rfloor}, \pi/2)$. Observons que cela entraîne, pour $K=1$, que $H _t \cong \lfloor \sqrt{t } \rfloor$. 
\end{remark}

\subsubsection{Frontières et vitesse de fuite}
On appelle {\em frontière} de la configuration, la configuration dans la zone $H \leq z_2 \leq H_+$. 
L'enjeu ici est de voir que si on choisit l'angle suffisamment petit, l'ensemble des frontières possibles est fini (et même très simple). Nous dirons qu'une configuration $\eta$  satisfaisant $H^+(\eta) - H(\eta) \leq 1$ est {\it équilibrée}. 
\begin{lemma}
\label{lem:petitangle}
Pour tout $K$, il existe un angle $\theta^*(K)$ tel que si $0 < \theta_0 < \theta^*(K)$, pour toute valeur de $\bfh$, toute configuration initiale $\eta_0$ équilibrée et tout $\bfx_0$ tel que $(\bfx_0,\eta_0,\theta_0) \in B_{\bfh}^*$,  la configuration $\eta_t$ est  équilibrée pour toute valeur de $t$.  

En outre,  si $0 < \theta_0 < \theta^*(K)$, pour toute valeur de $\bfh$, toute configuration initiale $\eta_0$ bien remplie et tout $\bfx_0$ tel que $(\bfx_0,\eta_0,\theta_0) \in B_{\bfh}^*$,  il existe $t_0 >0$ (ne dépendant que de la case vide la plus haute, $H^+(\eta_0)$) tel que la configuration $\eta_{t_0}$ est  équilibrée.

On peut prendre $\theta^*(K) = \arctan{\frac{1}{K(K-1)}}$ et donc en particulier $\theta^*(2) = \arctan{\frac{1}{2}}$. 
\end{lemma}

Une conséquence immédiate est que, si $\theta_0 < \theta^*(K)$, alors   $T$ est fini  sur  $B_{\bfh,\theta_0}^*$. 
\begin{proof} Une première remarque est que (dans tous les cas), si $H^+_n=H_n$, alors $H^+_{n+1} = H_{n+1} + 1 = H_n+1$.  Plus généralement, si on atteint la hauteur $H_n$  à une abscisse correspondant à une brique occupée ($\eta(\lfloor \bfx_{t-} \rfloor) = 1$ juste avant à l'instant $t$ où on atteint la hauteur $H_n$), alors on détruit cette brique et on redescend. 

Ensuite on affirme que si l'angle est assez petit et qu'on dépasse la hauteur $H_n$ (en passant à hauteur $H_n$ à une abscisse où une brique a été préalablement détruite),   on va détruire d'abord complètement cette rangée avant de toucher la rangée supérieure. Une assertion intermédiaire est que : s'il y a $b$ briques vides à l'étage $H$ la bille ne peut parcourir une distance horizontale supérieure à  $2b$  (en restant verticalement entre les hauteurs $H$ et $H+1$) sans toucher une brique. Cela entraîne que, si la bille parcourt une distance horizontale  supérieure à $2(1+2+\cdots+K-1) = K(K-1)$ en restant entre les hauteurs $H$ et $H+1$,  elle aura détruit toutes les briques de cet étage.  Or c'est ce qui arrive si   $\tan \theta < 1 / K(K-1)$. Ainsi, sans avoir à préciser la combinatoire, on déduit que si $\tan \theta < 1/ K(K-1)$, toutes les briques de la rangée $H_n$ sont détruites avant que la bille n'atteigne la hauteur $H_n +1$. 
\vip
Si la configuration initiale était bien équilibrée, la bille  rebondit alors sur la paroi horizontale d'une brique du niveau $H_n+1$ avant de redescendre et la configuration $\eta_{n+1}$ est bien équilibrée. Cet argument permet une preuve par récurrence élémentaire de la première assertion. 
\vip
Pour prouver la deuxième assertion, le raisonnement ci dessus  nous permet d'affirmer que les briques situées à hauteur inférieure à $H^+(\eta)$ vont être toutes détruites avant que la première brique à hauteur $H^+$ ne soit touchée, de telle manière qu'au moment où elle atteint $H^+$ puis redescend, on a bien $H^+_{t_0} -H_{t_0} \leq 1$.  
\end{proof}

\begin{figure}
\includegraphics[width=60mm]{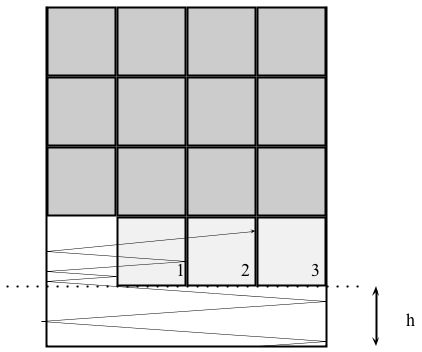}

  \caption{ Illustration de la pire situation possible pour $K=4$, avec un angle suffisamment petit pour que toutes les briques d'un étage soient malgré tout détruites en un seul passage. }
  \label{fig:presh=1/4}
\end{figure}

Ainsi, il apparaît que pour $\theta_0$ suffisamment petit, l'ensemble des configurations équilibrées est stable. En outre toute configuration initiale (bien remplie) devient équilibrée en un temps fini. Observons que ce résultat va nous permettre de travailler, mais ne prétend pas être optimal. 

 Pour $\theta < \theta^*(K)$, le nombre de configurations est fini (modulo la hauteur). C'est-à-dire qu'on peut paramétrer la configuration en indiquant la hauteur $H$ de la brique la plus basse et la liste (dans $\{0,1\}^K \setminus(0, \ldots, 0)$) des  briques présentes à l'étage $H $. Formellement, on peut écrire $\eta(H,\xi)$ où $H$ est un entier positif et $\xi \in  \conf^*_K=\{0,1\}^K \setminus(0, \ldots, 0)$ pour la configuration $\eta_{\bf z}=0$ pour $z_2<H$, $\eta_{\bf z}=\xi_{z_1}$ pour $z_2=H$ et  $\eta_{\bf z}=1$ pour $z_2>H$. 

Dans un premier temps,  on exploite simplement la forme particulière de la frontière pour donner une première estimation de la vitesse de fuite lorsque $\tilde \theta_0 \leq \theta^*(K)$. On peut montrer que  
si $\theta_0 < \theta^*(K)$, pour toute abscisse initiale, et si $H^+_0 - H_0 \leq 1$, le nombre de briques heurtées en un aller-retour est supérieur ou égal à $1$. En outre il est inférieur ou égal à $K$ (et même sous nos hypothèses, il est soit égal à $1$, soit égal à $1$ plus la totalité des briques restantes à l'étage $H$, mais ce n'est pas crucial à ce stade). On en deduit immédiatement que 
\begin{equation}
\frac{1}{K} \leq \frac{H _n}{n}  \leq 1. 
\end{equation}
En outre on peut relier  $T_n$ à $H _n$ en remarquant que soit $T = 2 (H+\bfh)/ \sin \theta_0$, soit $T= 2 (H +\bfh+1)/ \sin\theta_0$.  

\subsubsection{Identification des hauteurs}
On souhaite maintenant exploiter une certaine périodicité modulo $2K  \tan \theta$  pour ramener l'étude du casse-briques à celle d'un système dynamique sur un espace compact.   Le point est que si on a une hauteur sans briques de $2 K  \tan \theta$, on peut la "supprimer" sans changer la dynamique. Cela amène naturellement à regarder le système en considérant la hauteur modulo $2 K \tan \theta$. C'est un peu technique à exprimer, mais l'idée très simple est  exprimée par la figure~\ref{fig:tranche1}. Pour dire les choses précisemment, on va dire que l'application de premier retour du flot est inchangée si on retire une bande de hauteur $2 K  \tan \theta$ au flot. 
\vip

Rappelons d'abord l'invariance par translation suivante : si on translate tout le dessin d'un entier vers le bas, on ne change rien à la dynamique; Formellement, cela s'écrit, en posant ${\bf u} = (0,-p)$ et $(\eta+{\bf u}) ({\bf z}) = \eta({\bf z} - {\bf u})$ (cette configuration est vide sous la hauteur $p$),  pour tout $t>0$, 
$$
\Phi_t^{\bfh+1}(\bfx+{\bf u} , \eta + {\bf u}, \theta) = \Phi_t^{\bfh}(\bfx, \eta, \theta) + {\bf u},$$
et donc aux retours à la base, 
\begin{equation}
\label{translation}
\phi_{\bfh+1}(\bfx+{\bf u} , \eta + {\bf u}, \theta) = \phi_{\bfh}(\bfx, \eta, \theta) + {\bf u}.
\end{equation}
Cela entraîne qu'on peut toujours se ramener à une configuration avec $H=0$ quitte à changer $\bfh$. Cela revient à regarder le casse-briques depuis la "frontière". 

\vip

Une autre invariance, plus propre à ce modèle particulier est liée au fait qu'en l'absence de brique sur une tranche horizontale suffisamment grande, la trajectoire se répète. Plus précisement, on observe que, pour tout $\theta \in (0, \pi)$, si ${\bf x} =(x_1, x_2)$ satisfait $x_2+2K \tan \theta < H(\eta)$, alors, 
\begin{equation}
\label{decaluncran}
\Phi^\bfh_{2K \sin \theta} ({\bf x},\eta,\theta) = ({\bf x} + (0, 2K \tan \theta), \eta, \theta). 
\end{equation}

\begin{figure}
\includegraphics[height=80mm]{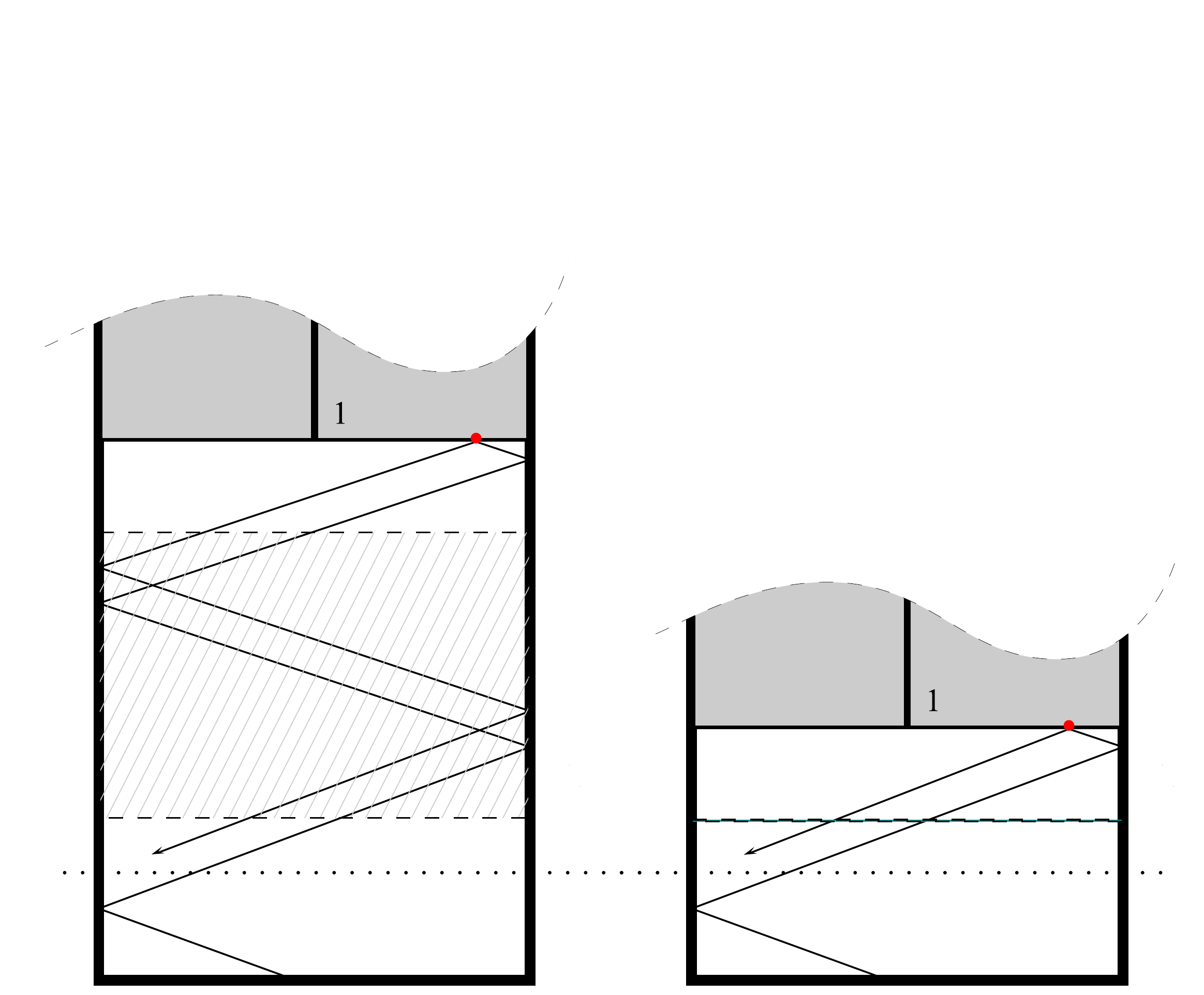}
\caption{Illustration de l'effet d'oublier une tranche de hauteur $2 K \tan{\theta}$~: vu d'en bas, l'impression est la même.} 
\label{fig:tranche1}
\end{figure}
Pendant un aller-retour à la base, cette identité à lieu à la montée et une identité symétrique au retour ; on en déduit l'égalité des applications de premier retour quand on supprime une tranche horizontale (vide) de hauteur $2K\tan{\theta}$ : 
\begin{lemma}
\label{lem:periode}
Soit $K \in \nn$.  Soient $\theta_0 < \theta^*(K)$ et $\bfh >0$.  Pour tout $(\bfx,\eta,\theta) \in B^*_{\bfh\bf,\theta_0}$, alors 
$(\bfx',\eta,\theta) = (\bfx - (0, 2K \tan \theta_0),\eta, \theta)  \in B^*_{\bfh+2K \tan \theta_0,\theta_0},$ et 
\begin{equation}
\label{periode}
 \phi_{\bfh + 2 K \tan \theta_0} (\bfx ',\eta, \theta) = \phi_\bfh(\bfx,\eta,\theta) 
 \end{equation}
 Enfin, si $p$ et $q$ sont deux entiers relatifs tels que $\bfh' = \bfh + 2 p K \tan \theta_0 + q >0$ et $q \leq H(\eta)$, 
\begin{equation}
\label{periodeplus}
 \phi_{\bfh'} (\bfx - (0, \bfh'-\bfh),\eta+q {\bf v}, \theta) = \phi_\bfh(\bfx,\eta,\theta) 
 \end{equation}
 \end{lemma}
\begin{proof}
On décompose la trajectoire entre la hauteur $-(\bfh+2K\tan\theta)$ et la hauteur $-\bfh$ (montée, aller retour depuis $-\bfh$, descente) en écrivant la propriété de semi-groupe. L'équation (\ref{decaluncran}) et sa symétrique pour  tout $\theta \in [\pi, 2\pi]$, si  ${\bf x} =(x_1, x_2)$ satisfait $ 2K \tan \theta  < x_2 < H$, permettent de conclure. 
 On obtient la seconde assertion en utilisant de manière répétée (\ref{translation}) et (\ref{periode})
\end{proof}

En particulier, cela va nous permettre de supprimer des tranches de hauteur $2K \tan \theta$ au fur et à mesure qu'on détruit des briques, de façon à "maintenir" la hauteur de la première brique dans un ensemble borné. Mais cela oblige à travailler avec le flot défini sur des domaines différents. Nous proposons ci-après une réparamétrisation du système qui permet d'exploitter cette remarque. 


\subsection{Reparamétrisation du système}

\label{sec:reparam}

Nous allons regarder notre système sous un angle un peu différent en profitant de différentes propriétés mises en évidence : 
\begin{itemize}
\item Les angles à la base ne prennent que deux valeurs. En créant une "copie" du système on se ramène à une seule valeur, passant d'une copie à l'autre quand l'angle change de valeur, c'est-à-dire lors des rebonds sur les parois verticales. C'est une transposition de l'idée, classique, qui permet de passer d'un billard à un flot géodésique sur une surface plate en "dépliant" le billard. 
\item Une fois l'angle initial  fixé, l'application de premier retour à la base est déterminée~: l'angle devient un paramètre et n'est plus à traiter comme une variable. 
\item Au contraire, pour se ramener à un système sur un compact, nous allons utiliser le lemme~\ref{lem:periode} pour maintenir la hauteur dans un intervalle borné. Pour le faire nous allons enlever des tranches de hauteur  $2K \tan \theta$. Mais cette opération modifie la hauteur $\bfh$ du flot à considérer et donc l'application à utiliser pour suivre la dynamique. Nous traiterons donc l'information restant sur la hauteur comme une variable. 
\end{itemize}
Ainsi, cela nous amène à définir un nouveau système dynamique pour lequel l'espace de phase est différent. Ce système est relié au système initial par une espèce de conjugaison. Le nouveau système est plus facile à étudier et l'information concernant une orbite du système initial peut être reconstruite à partir de l'état initial et de l'orbite dans le nouveau système.  

\subsubsection{Cas général}
Fixons $\theta_0 < \theta^*(K)$. On va reparamétrer le système 
par une famille de copies du  tore $\tore \times \{0,1\}^K$  indéxée par les configurations frontières $\cC^*_K=  \{0,1\}^K \setminus \{(0, \ldots, 0)\}$~: $x$ représente l'abscisse ($x_1$ si $\theta =   \theta_0$ et $2K-x_1$ si $\theta = \pi -  \theta_0$) divisée par $2K$, $h$ la hauteur $H+\bfh$ divisée par $ 2 K \tan\theta_0$ modulo $1$ et $\xi \in \cC^*_K$ est la configuration de la ligne à hauteur $H$. 
Ainsi, on associe à un état  $(\bfx,\eta, \theta)$ appartenant à la base (et issu d'une condition initiale $\theta_0$) un point $\Psi({\bf x}, \eta,\theta) = (x,h,\xi)$ de $\tore \times \cC^*_K$ en posant : 
$$ 
\begin{array}{lccc}
\Psi: & B_\bfh & \to & \tore \times \{0,1\}^K \\
& (\bf x,\eta, \theta) & \mapsto &  (x,h,\xi), 
\end{array}
$$
avec, 
\begin{eqnarray*}
\left\{ \begin{array}{lll}
x &=&  \left\{ \begin{array}{ll}  x_1 / 2K,  & \hbox{ si } \theta =  \theta_0 \\  1-x_1/2K,   & \hbox{ si } \theta=\pi-  \theta_0,  \end{array} \right. \\
h &=&  (H(\eta)+ \bfh) / (2 K \tan\theta_0) \hbox{ modulo } 1,  \\
\xi &=& (\eta(k,H))_{k=0,\ldots, K-1}. 
\end{array} \right. \end{eqnarray*}


\vip 
Pour chaque  $\bfh$, l'ensemble $B^*_{\bfh,\theta_0}$ des états réguliers est envoyé sur un sous ensemble de $[0,1] \times \{h\} \times \cC_K^*$ de mesure de Lebesgue pleine.  
Observons que l'application  $\Psi$ n'est pas inversible. Cependant elle préserve suffisamment d'information pour permettre de suivre l'orbite d'un point. Observons tout d'abord que  
\begin{lemma}
\label{quotient}
Soient $\theta \in \{\theta_0, \pi - \theta_0\}$,  $\eta$ et $\eta'$  deux configurations bien équilibrées, $\bfh$ et $\bfh'$ réels positifs, $\bfx \in \Delta_{\bfh}$  et $\bfx' \in  \Delta_{\bfh'}$ tels que  $\Psi(({\bf x}, \eta, \theta)) = \Psi(({\bf x'}, \eta', \theta))$ ; alors 
$$  H(\phi_\bfh(\bfx,\eta,\theta)) - H( \bfx,\eta,\theta) = H(\phi_{\bfh'}(\bfx',\eta',\theta)) - H(\bfx',\eta',\theta)$$
et
$$     \Psi( \phi_\bfh(\bfx,\eta,\theta)) = \Psi(\phi_{\bfh'}(\bfx',\eta',\theta)).$$
\end{lemma}
\begin{proof}
L'égalité $\Psi(({\bf x}, \eta, \theta)) = \Psi(({\bf x'}, \eta', \theta)) = (x,h,\xi)$, pour $\theta \in \{\theta_0, \pi - \theta_0\}$ et pour deux configurations $\eta$ et $\eta'$ bien équilibrées,  est équivalente à l'existence de ${\bf v }= (0,v)$ et d'un entier $k$  tels que $x'_1 = x_1$, $\eta = \eta'+{\bf v}$ et  $\bfh' - \bfh = x_2 - x'_2 = v + 2kK \tan \theta_0$. 
Dès lors, on peut utiliser le lemme~\ref{lem:periode} pour conclure. 
\end{proof}
En conséquence,  $\theta_0$ étant fixé, pour tout  $(x,h,\xi) \in \tore \times \cC^*_K$, 
 on peut trouver un  unique représentant $({\bf x}, \eta, \theta) \in B_\bfh$ avec $\bfh \in [0 , 2K \tan \theta_0]$ et $H(\eta) =0$   tel que $\Psi(({\bf x}, \eta, \theta)) = (x,h,\xi)$. Concrètement, on fixe $\bfh = (2K  \tan \theta_0)  h$,  on note  $\eta^\xi (z_1,z_2) = 1 - \indiq_{\{\xi(z_1) =0, z_2=0\}}$,  et on pose 
$$\Psi_*^{-1} (x,h,\xi) =    \left\{ \begin{array}{ll}  ((2Kx,-\bfh), \eta^\xi, \theta_0)  & \hbox{ si } x \leq 1/2  \\   ((1-2Kx,-\bfh),  \eta^\xi, \pi-\theta_0)    & \hbox{ si }  x >1/2.  \end{array} \right.  
$$
Cela va nous permettre de construire une application $\varphi$ sur $ \tore \times \{0,1\}^K$ satisfaisant $\varphi  \circ \Psi   = \Psi \circ \phi$ et une application $\Delta H$ telles que si   
$\Psi(\bfx,\eta,\theta) = (x,h,\xi)$ et $\varphi^n(x,h,\xi) = (x',h',\xi')$, alors, 

$$\phi^n( \bfx,\eta,\theta)   =     \left\{ \begin{array}{ll}  ((2Kx',-\bfh), \eta', \theta_0)  & \hbox{ si } x' \leq 1/2  \\   ((K-2Kx',-\bfh),  \eta', \pi-\theta_0)    & \hbox{ si }  x' >1/2, \end{array} \right.  
$$
où la frontière de $\eta'$ est donnée par $\xi'$ et sa hauteur par 
$$H(\eta') = H(\eta) + \sum_{k=0}^{n-1} \Delta H \circ \varphi^n(x,h,\xi).$$

\vip

Nous allons  maintenant construire une  fonction $\varphi$ sur $\tore \times\cC^*_K$ telle que 
$ \varphi  \circ \Psi= \Psi \circ \phi.$ On peut le faire de telle manière que : 

\begin{proposition}
\label{prop:conjug}
Soit $\theta_0 < \theta^*(K)$. Posons $\alpha = 1/(2 K \tan{\theta_0})$. Il existe une famille de partitions finies $\cP_h = \{I^h_i, i \in {\mathcal Q} \}$ de $\cercle  \times \cC^*_K$, indexée par $h \in [0,1)$, en intervalles et une famille $\{ (\gamma_i, \epsilon_i, \xi_i) , i \in {\mathcal Q} \} $ d'éléments de $\cercle \times \{0,1\} \times \cC^*_K$  telles que la fonction $\varphi$ définie, pour tout $i \in {\mathcal Q}$,  tout $h \in [0,1)$ et tout $(x, \xi) \in I_i^h$, par 
$$\varphi(x,h,\xi) = (x+2h+\gamma_i, h +  \epsilon_i \alpha , \xi_i)$$
satisfasse, pour tout $\bfh \in \rr_+$, et tout $(\bf x,\eta, \theta) \in B_{\bfh}$ avec $\theta \in \{\theta_0, \pi-\theta_0\}$, 
$$ \varphi  \circ \Psi  ({\bf x},\eta, \theta) = \Psi \circ \phi_{\bfh}  (\bf x,\eta, \theta).$$
\end{proposition}
Une preuve directe de ce résultat, consistant essentiellemnt à expliciter l'application $\varphi$ est donnée en section~\ref{sec:expressionvarphi}. Nous donnons ici un aperçu rapide des arguments qui permettent d'assurer les propriétés annoncées. 
\begin{proof}
Remarquons  que les contraintes combinatoires imposées font que seuls deux temps de retour sont possibles, selon (i) qu'on tape la paroi inférieure d'une brique de la première rangée (et que donc on redescend tout de suite), ou (ii) qu'on attaque la première rangée à un endroit où il n'y a pas de briques, dans quel cas, on va  détruire intégralement la rangée avant d'aller taper une brique de la rangée suivante sur sa paroi inférieure pour redescendre. Analysons les deux situations possibles : 
\begin{itemize}
\item La bille heurte une brique sur sa paroi horizontale inférieure et redescend ; on a alors modifié la configuration $\xi$ ; en général, la hauteur $h$ reste inchangée, sauf si la brique touchée était la dernière de sa rangée, dans quel cas, $h$ est incrémenté de $\alpha=1/2K \tan{\theta_0}$ et la nouvelle configuration frontière est pleine ; on observe que dans tous les cas, le déplacement horizontal est (vu sur le cercle) le même, de $\gamma(h) = 2h$. 
\item La bille arrive à la hauteur de la brique la plus basse en une case vide ; alors notre condition sur $\theta_0$ garantit qu'elle va heurter toutes les briques à cette hauteur (sur leur paroi verticale) avant d'en heurter une de la rangée supérieure. Il est clair que la durée de l'aller-retour est indépendante de la configuration et de l'ordre dans lequel elle heurte les différentes briques de la rangée puisque le module de la vitesse verticale est constant et que la distance verticale parcourue est $2h+1$. En revanche, le déplacement horizontal, lui,  dépend de manière un peu sophistiquée de la configuration $\xi$ et de la paroi de la première brique touchée. 
Enfin, relevons que la configuration image dépend de la brique heurtée sur la rangée supérieure. Mais toutes ces données combinatoires sont constantes par morceaux en fonction de l'abscisse de départ. 
\end{itemize}

Nous ne chercherons pas  formaliser plus ici. Il ressort de cette description que pour chaque hauteur $h$, la réunion d'intervalles correspondant aux abscisses et aux configurations peut être partitionnée en un nombre fini d'intervalles sur lesquels  $\varphi(x,h,\xi) = (x+2h+\gamma_i, h +   \epsilon_i \alpha, \xi_i)$ où $\gamma_i, \alpha_i$ et $\xi_i$ ne dépendent que de l'intervalle de la partition (elle-même dépendant de $h$) dans laquelle se situe $(x,\xi)$. Insistons en particulier sur le fait que $\epsilon_i= 0$ ou $1$. 
\end{proof}
En général, cette application peut être vue comme une transvection sur un produit fini de tores $\tore$. C'est-à-dire qu'il existe des partitions finies de chacun  des tores telles que sur chaque morceau de la partition, l'application soit la composée de la transvection $(x,h,\xi) \mapsto  (x+2h, h,\xi)$ et d'une "translation" de vecteur $  (\gamma_i, \epsilon_i \alpha, \xi_i)$ changeant éventuellement de tore. Mais sa forme particulière nous permet d'être plus précis ; en effet, la seconde coordonnée évolue suivant l'orbite $\cR_\alpha(h_0) = \{ h_0 + n\alpha \; ; \; n \in \nn\}$ d'une rotation d'angle $\alpha$ fixé partant de $h_0$. Lorsque $\alpha$ est rationnel, cet ensemble est fini : l'espace se décompose en réunion finie d'intervalles invariantes par la dynamique (indexées par $h_0$). La restriction de $\varphi$ à chacune de ces fibres, est une translation d'intervalles. 
\begin{cor}
Si $\alpha=1/2K\tan{\theta_0}$ est rationnel, alors la restriction de l'application $\varphi$  à la réunion finie d'intervalles $[0,1] \times \cR_\alpha(h_0) \times \{0,1\}^K$ est une translation d'intervalles. 
\end{cor}
\begin{proof}
C'est immédiat dans la mesure où $\epsilon_i \in \{0,1\}$. Il résulte que l'ensemble des hauteurs atteint est discret et que donc une fois $h_0$ fixé, $\varphi$ restreinte à l'orbite est une translation d'intervalles. \end{proof}
Observons que si,  en outre,  $h_0$ est rationnel, alors le système est périodique puisque les vecteurs de translation sont eux-même rationnels. 

\vip

Ce résultat donne l'impression qu'on devrait pouvoir comprendre en détail le comportement du casse-briques quand $\alpha$ est rationnel. Cependant, il reste une difficulté importante~: $\varphi$ n'est pas bijective. Nous sommes naturellement amenés à nous interroger sur son ensemble limite~: $\cap_{n\geq0} \varphi^n(\tore \times \cC^*_K)$ (et sa décomposition en sous ensembles invariants). Celui-ci pourrait être une réunion d'intervalles ; mais ce n'est pas le cas en général : à priori, cela peut aussi être un ensemble de Cantor de l'intervalle (voir reférence \cite{bos1}). Dans les deux cas, il reste à s'interroger sur la minimalité (difficulté plutôt combinatoire ici) puis sur l'ergodicité de la mesure de Lebesgue pour le  système dynamique induit par $\varphi$ sur l'ensemble limite. Ainsi, pour être plus précis, il faut comprendre mieux la combinatoire du système. Même dans le cas agréable où l'ensemble limite est une réunion d'intervalle, où la restriction de l'application est  un échange d'intervalles et donc où la mesure de Lebesgue correctement renormalisée est invariante,  son  ergodicité n'est pas garantie en général (voir référence \cite{kea1}) et une compréhension fine de la dynamique s'impose.

Nous ne chercherons dans  cette première étude à entrer dans le détail que dans le cas $K=2$ et $\tan{\theta} = 1/4$. Mentionnons cependant que ces difficultés combinatoires pourraient être levées et assurer, dans la situation irrationelle : 
\begin{conjecture}
Si $\alpha$ est irrationnel, alors $\varphi$ est (i) minimale, (ii) uniquement ergodique sur son ensemble limite.  
\end{conjecture}

\begin{remark}
\label{rem:geom} Lorsque $\bfh$ est rationnel, il est possible de faire une construction géométrique semblable à celle qui est communément faite pour les billards polygonaux rationnels : on peut en quelque sorte "déplier" le billard autour des parois sur lesquelles ont lieu les rebonds. C'est facile (et en partie  fait) le long du bord du domaine. C'est plus délicat le long des obstacles parce qu'il faut traduire le fait que l'obstacle heurté disparaît. Mais c'est possible en utilisant l'astuce permettant de se limiter à un ensemble compact de hauteurs ; le défaut de bijectivité faisant apparaître une {\em surface plate  branchée} compacte sur laquelle le casse-briques devient le flot géodésique.  Ce point de vue, s'il a pu guider notre intuition nous semble trop technique à développer par rapport à ce qu'il apporte pour l'instant. 
\end{remark}
\subsubsection{Cas $K=2$ et $\tan{\theta} = 1/4$}
Considérons maintenant le cas où $\tan \theta$ est rationnelle. Posons $\tan \theta = \frac{p}{q}$. Il vient $\alpha = \frac{q}{4p}$. Regardons dans quels cas $\alpha$ est entier. Si $p$ et $q$ sont premiers entre eux, cela n'arrive que si $p=1$ et $4$ divise $q$, i.e. si $\tan \theta = \frac{1}{4 q'}$. Dans ce cas l'application est particulièrement simple parce que $\alpha =0 $ modulo $1$ ne joue plus aucun rôle (et la partition est plus simple) ; elle est indépendante de la valeur de $q'$.  Pour chaque valeur de $h$, on note $I_h = [0,1] \times \{h\} \times \cC^*_2$ qu'on regarde comme trois copies d'un cercle (en identifiant $1$ et $0$ sur chaque copie de l'intervalle). On distingue en outre les positions correspondant à des orbites qui vont atteindre la hauteur $0$ à droite ($x<1/2$ ou $x > 3/2$) et à gauche ($1/2 < x < 3/2$) en posant  $I^h_1 = \cR_{-2h}([3/2,1/2])$ et $I^h_2 = \cR_{-2h}([1/2,3/2])$, modulo 1 ; ce sont des réunions de trois demi-cercles. On peut alors écrire~: \begin{eqnarray*}
\varphi(x,h,(1,1)) &=& \left\{ \begin{array}{ll} (x+2h, h, (0,1)),  & \hbox{ si } x \in I_1^h \\  (x+2h, h, (1,0)),  & \hbox{ si } x \in I_2^h  \end{array} \right. \\
\varphi(x,h,(0,1)) &=& \left\{ \begin{array}{ll}  (x+2h+1/2, h, (1,0) ) & \hbox{ si } x \in I_1^h   \\   
  (x+2h, h, (1,1)),  & \hbox{ si } x \in I_2^h   \end{array} \right. \\
\varphi(x,h,(1,0)) &=& \left\{ \begin{array}{ll} (x+2h, h, (1,1)),  & \hbox{ si } x \in I_1^h \\   
(x+2h+1/2, h, (0,1) ) & \hbox{ si } x \in I_2^h.  \end{array} \right. 
\end{eqnarray*}
On observe, entre autres choses, que comme attendu, la deuxième composante, $h$, ne bouge pas. La figure~\ref{fig:atantheta4} donne une idée des ensembles limites pour chaque valeur de $h$. Nous allons nous concentrer sur ce qui se passe pour $h$ proche de $0$.  
\begin{figure}
\label{fig:enslim}
 \includegraphics[height=120mm]{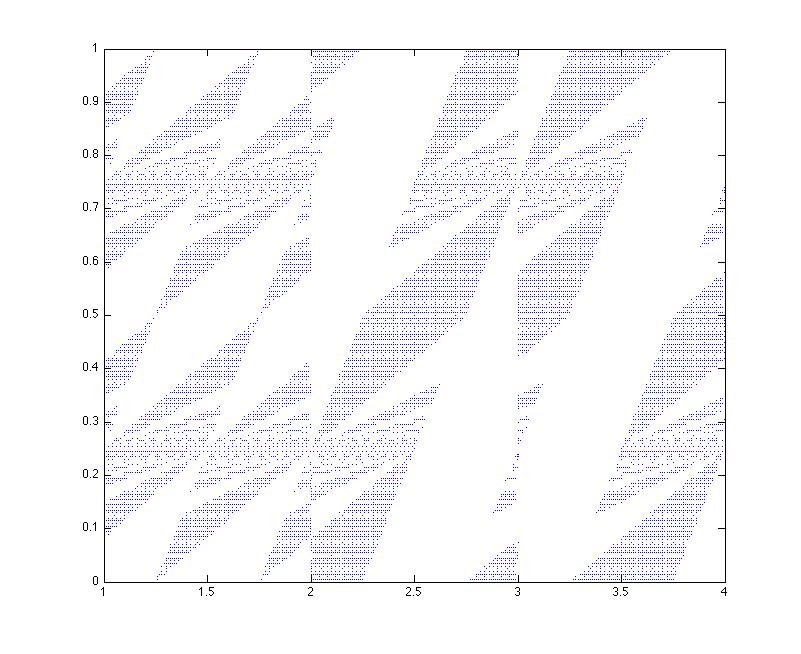}
\caption{On a tracé ici les orbites pour diverses  valeurs de $h_0$ : pour chaque $h_0$ et pour chaque position initiale ($x_0,\xi_0$), on a tracé l'orbite entre la 1000ème et la 2000ème itération, afin de faire apparaître l'ensemble limite. On distingue les trois tores correspondant à $\xi = (1,1), (0,1)$ et $(1,0)$ disposés cote à cote (intervalles $[1,2), [2,3)$ et $[3,4)$, respectivement). On voit apparaitre, près de $h=0$, l'ensemble limite tel que proprement décrit ci-après. La situation pour $h=1/4$ et $h=3/4$, rationnels est spécialement simple. Mais elle parait confuse dans leur voisinage. Voir aussi les figures~ \ref{fig:vitesse} et \ref{fig:presh=1/4}. 
}
\label{fig:atantheta4}
\end{figure}
Considérons d'abord le cas où $h=0$ (mod $1$). Il vient 
\begin{eqnarray*}
\varphi(x,0,(1,1)) &=& \left\{ \begin{array}{ll} (x, 0, (0,1)),  & \hbox{ si } x \in I_1 \\  (x, 0, (1,0)),  & \hbox{ si } x \in I_2  \end{array} \right. \\
\varphi(x,0,(0,1)) &=& \left\{ \begin{array}{ll} (x+1/2, 0, (1,0) ) & \hbox{ si } x \in I_1 \\  (x, 0, (1,1)),  & \hbox{ si } x \in I_2  \end{array} \right. \\
\varphi(x,0,(1,0)) &=& \left\{ \begin{array}{ll} (x, 0, (1,1)),  & \hbox{ si } x \in I_1 \\   
(x+1/2, 0, (0,1) ) & \hbox{ si } x \in I_2.  \end{array} \right. 
\end{eqnarray*}
Dans ce cas trivial, on s'aperçoit que 6 intervalles (de longueur $1/2$) sont envoyés les uns sur les autres. Si on les note $I_k^{conf}$, on peut représenter la dynamique par le graphe~:  
$$I_1^{(1,0)}  \to I_1^{(1,1)}  \to  I_1^{(0,1)} \leftrightarrow I_2^{(1,0)}       \leftarrow  I_2^{(1,1)}     \leftarrow  I_2^{(0,1)} . $$
Considérons maintenant $h>0$ pas trop grand. Nous affirmons qu'alors,

\begin{lemma}
\label{lem:induite}
Soit $h < 1/10$ fixé. (i) L'induite $\overline{\varphi}$ de $\varphi$ sur l'intervalle $G_h= [3/4+5h,1/4-h] \times \{h\} \times \{(0,1)\} $ s'écrit : 

$$\overline{\varphi}(x) =  \left\{ \begin{array}{ll}  x+4h  & \hbox{ si }  3/4+5h  \leq x < 1/4-5h  \\ x+10h-1/2 & \hbox{ si }   1/4-5h \leq  x \leq 1/4-h . \end{array} \right. $$
En outre, (ii) l'ensemble $\cJ_h = \cup_{k=0}^{4} \varphi^k(G_h)$ est invariant par $\varphi$ et (iii) $\varphi^3(I_h) \subset \cJ_h$. 
\end{lemma}
\begin{figure}
\begin{center} 
\includegraphics[width=100mm]{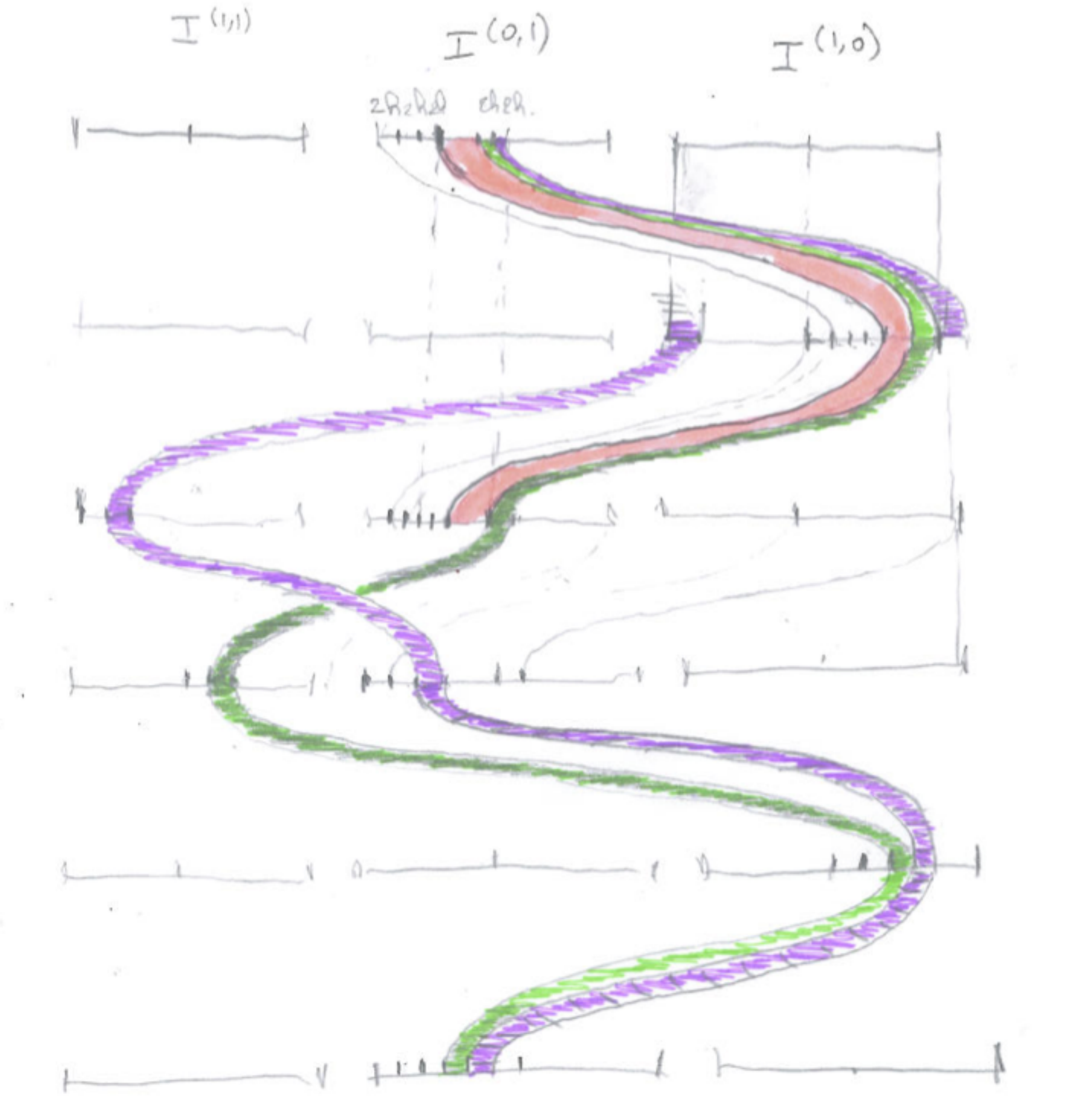}
\end{center}
  \caption{ Une vue de $\varphi$ et de l'induction. }
  \label{fig:hierarchie}
\end{figure}


\begin{proof}
Pour suivre la dynamique hors de $G_h$, il faut diviser l'intervalle initial en 3 morceaux. Les deux premiers se séparent du troisième à la première itération. Le premier revient dans $G_h$ dès l'itération suivante (mais pas le deuxième).  Les deux derniers se "recollent" (miraculeusement~?) à la quatrième itération avant de revenir dans $G_h$ à la cinquième. En écrivant les intervalles vus sur le cercle on obtient, sous l'action de $\varphi$ (on omet d'écrire la composante $\{h\}$ qui reste inchangée)~: 
\begin{itemize}
\item $[\frac{3}{4}+5h, \frac{1}{4}-5h] \times \{(0,1)\} \to [\frac{1}{4}+7h, \frac{3}{4}-3h] \times \{(1,0)\} \to [\frac{3}{4}+9h, \frac{1}{4}-h] \times \{(0,1)\}$
\item $[ \frac{1}{4}-5h, \frac{1}{4}-3h] \times \{(0,1)\} \to [\frac{3}{4}-3h, \frac{3}{4}-h]\times \{(1,0)\} \to  [\frac{1}{4}-h, \frac{1}{4}+h] \times \{(0,1)\}$ 
$\to [ \frac{1}{4}+h  , \frac{1}{4}+3h] \times \{(1,1)\}  \to [ \frac{1}{4}+3h  , \frac{1}{4}+5h] \times \{(1,0)\} \to [ \frac{3}{4}+5h  , \frac{3}{4}+7h] \times \{(0,1)\} \subset G_h  $
\item $ [\frac{1}{4}-3h, \frac{1}{4}-h] \times \{(1,0)\} \to [\frac{3}{4}-h, \frac{3}{4}+h] \times \{(0,1)\} \to [\frac{3}{4}+h, \frac{3}{4}+3h] \times \{(1,1)\} \to [\frac{3}{4}+3h, \frac{3}{4}+5h] \times \{(0,1)\} \to  [\frac{1}{4}+5h, \frac{1}{4}+7h] \times \{(1,0)\}  \to  [\frac{3}{4}+7h, \frac{3}{4}+9h] \times \{(0,1)\} \subset G_h$
\end{itemize}
Ainsi, $\varphi^3([\frac{3}{4}+5h, \frac{1}{4}-5h]) =  [\frac{3}{4}+9h, \frac{1}{4}-h]$,  $\varphi^5([ \frac{1}{4}-5h, \frac{1}{4}-3h]) = [ \frac{3}{4}+5h  , \frac{3}{4}+7h]$  et $\varphi^5([\frac{1}{4}-3h, \frac{1}{4}-h]) =  [\frac{3}{4}+7h, \frac{3}{4}+9h]$.  L'assertion (ii) découle immédiatement de cette analyse. On prouve l'assertion (iii) en analysant simplement l'action de $\varphi$ sur $I_h \setminus G_h$. 
\end{proof}

\subsection{Vitesse de fuite}
Nous allons maintenant montrer un résultat permettant de calculer la vitesse de fuite sous une hypothèse d'ergodicité. Nous verrons ensuite comment le lemme ~\ref{lem:induite} permet de l'appliquer dans le cas $K=2$, $\tan{\theta} = 1/4$, $h$ petit et ainsi prouver le théorème \ref{th:vitesse}. 
\subsubsection{Cas général}
Nous avons vu que sous la condition $\theta < \theta^*(K)$, la vitesse de fuite est comprise entre $1/K$ et $1$. Pour préciser le contrôle de la vitesse de fuite, on introduit $\Delta H = H \circ \phi - H$ qui représente l'incrément de la hauteur pendant un aller-retour. On distingue trois types d'allers-retour~: 

\vip

\noindent
(i)  ceux pour lesquels on heurte une brique à hauteur $H(\eta)$ sur son arête horizontale, ne détruisant que celle là et sans incrémenter $H$ (parce qu'il en reste d'autres) ; alors, $\Delta H =0$ et $T=2 (H+\bfh) / \sin{\theta}$ ;

\noindent
(ii) ceux  pour lesquels on heurte une brique à hauteur $H(\eta)$ sur son arête horizontale, alors que c'est la dernière de la ligne (alors $H$ est incrémenté de $1$) ; alors,  $\Delta H =1$ et $T=2 (H+\bfh) / \sin{\theta}$~;

\noindent
(iii) ceux pour lesquels on heurte les briques de niveau $H(\eta)$ sur une paroi verticale ; alors, puisque $\theta<\theta^*(K)$, on détruit toutes les briques de ce niveau et on finit par rebondir sur la paroi horizontale d'une brique de niveau $H+1$~; alors, $\Delta H =1$ et $T=2 (H+\bfh+1) / \sin{\theta}$. 

\vip

Cela fournit une partition de $B^*$ (de $B^*_{\bfh,\theta}$ pour chaque $\bfh$ et $\theta$) en trois morceaux, notés respectivement~:  $U_{0}$, $U_{+}$ et $U_{++}$.   On écrit naturellement la hauteur et la durée après $n$ retours comme des cocycles au dessus de $\phi$ : 
$$ H_n -H_0=  \sum_{k=0}^{n-1} \Delta H \circ \phi^k  =  \sum_{k=0}^{n-1}\indiq_{U_+ \cup U_{++}} \circ \phi^k .$$
On observe que 
$$T= \frac{2}{\sin{\theta}}   ((H+\bfh) \indiq_{U_0 \cup U_+} + (H+\bfh +1) \indiq_{U_{++}}) =  \frac{2}{\sin{\theta}}   (H +\bfh + \indiq_{U_{++}}). $$
Donc, 
\begin{eqnarray*}
 \tau_n &:=&   \sum_{k=0}^{n-1} T \circ \phi^k   =  \frac{2}{\sin{\theta}}   (n (H_0+\bfh) +  \sum_{k=0}^{n-1}  \left( (n-k) \indiq_{U_+ \cup U_{++}} +    \indiq_{U_{++}} \right) \circ \phi^k  ).
\end{eqnarray*}
Pour exploiter notre reparamétrisation du système,  on observe maintenant qu'une fois $\theta_0$ et $\bfh$ fixés,  la fonction $\Delta H$ est mesurable par rapport $ \Psi$. Autrement dit, si on fixe $A_0$, $A_+$ et $A_{++}$ images réciproques respectives  de $U_0$, $U_+$ et $U_{++}$ par $\Psi$, on a 
$$\Delta H ({\bf x}, \eta,\theta) = \indiq_{U_+ \cup U_{++}}({\bf x}, \eta,\theta)  = \indiq_{A_+ \cup A_{++}}(\Psi({\bf x}, \eta,\theta)) $$
On peut écrire explicitement ces ensembles. Plus généralement, dans la mesure ou $\indiq_U \circ \phi^k = \indiq_A \circ \Psi \circ \phi^k = \indiq_A \circ \varphi^k \circ \Psi$, cela permet d'écrire $H_n/n$ et $\tau_n / n^2$ comme dérivées de sommes ergodiques pour $\varphi$~: 

\begin{eqnarray}
\label{sommeergodique}
\frac{H_n}{n} &= & \frac{1}{n} \sum_{k=0}^{n-1}\indiq_{A_+ \cup A_{++}} \circ \varphi^k \circ \Psi \\
\label{sommeergodique2}
 \frac{\tau_n }{n^2} &=&     \frac{2}{n^2 | \sin{\theta} |}  (nH_0+   \sum_{k=0}^{n-1}  \left( (n-k) \indiq_{A_+ \cup A_{++}} +    \indiq_{A_{++}} \right) \circ \varphi^k  ).
 \end{eqnarray}
\begin{proposition}
\label{prop:fuite}
Soit $\mu$ une mesure de probabilité  invariante ergodique pour $\varphi$ et $\overline{\mu} = \Psi^*\mu$.  Alors, $\overline{\mu}$-p.s., 
$$ \lim_{n \to \infty} \frac{H_n }{n} =    \mu({A_+ \cup A_{++}})$$
et
$$\lim_{t \to \infty} \frac{H_t}{\sqrt{t}}  = \sqrt{\mu({A_+ \cup A_{++}}) \sin{\tilde{\theta}_0}}.$$
\end{proposition}
\begin{proof}
La première convergence découle d'une application directe du théorème ergodique à l'expression (\ref{sommeergodique}). Appliqué à (\ref{sommeergodique2}), le théorème ergodique assure que 
$$ \lim_{n \to \infty} \frac{\tau_n }{n^2}  =   \frac{  \mu({A_+ \cup A_{++}}) }{\sin{\tilde{\theta}_0}}. $$
Pour prouver la seconde convergence, on note $n = n(x,t)$  l'unique entier tel que $\tau_n \leq t < \tau_{n+1}$. De l'écriture 
  $$\frac{H_t}{\sqrt{t}} = \frac{H_n + (H_t - H_n)}{\sqrt{\tau_n + (t-\tau_n)}} = \frac{ \frac{1}{n} H_n + \left(\frac{H_t-H_n)}{n}\right)}{\sqrt{ \frac{\tau_n}{n^2} + \left(\frac{t-\tau_n}{n^2}\right)}},$$
on déduit 
$$\lim_{t \to \infty} \frac{H_t}{\sqrt{t}} = \frac{\mu({A_+ \cup A_{++}})}{\sqrt{   \frac{  \mu({A_+ \cup A_{++}}) }{\sin{\tilde{\theta}_0}} }} = \sqrt{\mu({A_+ \cup A_{++}}) \sin{\tilde{\theta}_0}}.$$
En effet, le point important est que $n(x,t)$ tend vers l'infini avec $t$ uniformément en $x$. Ensuite on observe simplement que : $|H_t - H_n| \leq |H_{n+1}-H_n| \leq 1$ et  $|t-\tau_n| \leq |\tau_{n+1} - \tau_n| \leq 2 H_n,$ de telle manière que $\frac{1}{n} |H_t - H_n|  \to 0$ et $\frac{|t-\tau_n|}{n^2} \leq \frac{H_n}{n^2} \to 0$. Cela permet d'évaluer la limite en $t$ connaissant les deux autres en $n$.  \end{proof}


\subsubsection{Preuve du Théorème \ref{th:vitesse}} 
Nous sommes maintenant en mesure de prouver le résultat annoncé~: il s'agit essentiellement de s'assurer de l'ergodicité de la mesure induite par la mesure de Lebesgue par $\Psi$ et d'expliciter les ensembles $A_+$ et $A_{++}$ pour calculer leur mesure. 
\begin{proof} Le lemme~\ref{lem:induite}  montre que, pour chaque $h<1/10$ irrationnel fixé,          l'application induite $\overline{\varphi}$ de $\varphi$ sur l'intervalle $G_h$  est conjuguée à une rotation d'angle $\overline{\alpha} = \frac{4h}{1/2-6h}= \frac{8h}{1-12h}$ qui est irrationnel dès que $h$ l'est. Puisque  $\overline{\alpha}$ est irrationnel, la mesure de Lebesgue est la seule mesure invariante pour $\overline{\varphi}$. L'application $\varphi$ restreinte à l'ensemble invariant $\cJ_h$ est un échange d'intervalles. Soit  $\mu$ une mesure invariante pour $\varphi_{| \cJ_h}$.  Sa trace sur $G_h$ est $\overline{\varphi}$-invariante~; c'est donc la mesure de Lebesgue sur $G_h$. L'invariance de $\mu$  par $\varphi$ composée de translations entraîne alors que $\mu$ est la mesure de Lebesgue sur $\cJ_h $. Celle-ci se trouve ainsi être l'unique mesure invariante pour $\varphi_{| \cJ_h }$. Cela nous permet d'affirmer qu'elle est ergodique. Ainsi, la restriction de la mesure de Lebesgue à $\cJ_h$ est une mesure $\varphi$-invariante et ergodique.

Rappellons que $\cJ_h$ s'écrit~:
$$\cJ_h= ([\frac{3}{4} + h, \frac{3}{4} + 3h] \cup [\frac{1}{4}+ h, \frac{1}{4}+ 3h]) \times \{(1,1)\} \cup  [\frac{3}{4} + 3h, \frac{1}{4}+h] \times \{(0,1)\} \cup [ \frac{1}{4}+3h, \frac{3}{4}+h] \times \{(1,0)\}.$$
Sa mesure totale est : $2h+2h+(1/2-2h)+(1/2-2h) = 1$. On  observe que la hauteur est incrémentée (de un) entre deux retours à la base (c'est-à-dire qu'on est dans $A_+ \cup A_{++}$) si et seulement si on est dans une configuration asymétrique, soit $(0,1)$, soit $(1,0)$.  Un calcul direct montre que la mesure de Lebesgue de $A_+ \cup A_{++} =\cJ _h\cap ( I^{(0,1)} \cup I^{(1,0)})$ est $1-4h$. Ainsi, en vertu de la proposition~\ref{prop:fuite}, il existe un ensemble de mesure de Lebesgue pleine dans $G_h$ tel que  pour tous les $(\bfx, \eta, \theta)$ tels qu'il  $\psi(\bfx, \eta, \theta)$ est dans cet ensemble, 
$$\lim_{n \to \infty} \frac{H_n}{n} = 1-4h, $$
 et 
 $$\lim_{t \to \infty} \frac{H_t}{\sqrt{t}} = \sqrt{(1-4h)/\sqrt{2}}.$$

Supposons maintenant $h<1/10$ rationnel. Le système est alors périodique ; il y a autant de mesures invariantes que d'orbites et la mesure de Lebesgue n'est pas ergodique : notre argument ne fonctionne plus. La répartition des orbites entre les intervalles $I^{(0,1)} \cup I^{(1,0)}$ et $I^{(1,1)}$ pourrait alors dépendre du point de départ.  Il se trouve ici qu'une analyse plus détaillée de la combinatoire montre que le résultat reste valable : en effet, toutes les orbites de $\overline{\varphi}$ passent la même proportion du temps dans les intervalles  $[ \frac{1}{4}-5h, \frac{1}{4}-3h] \cup [\frac{1}{4}-3h, \frac{1}{4}-h]$ car celui-ci étant précisément de longueur $4h$, on y passe une fois par tour.

Pour conclure complètement, il reste à relever que si l'état initial n'est pas envoyé dans $\cJ_h$  par $\Psi$ ou si la configuration initiale n'est pas bien équilibrée, il suffit d'attendre un temps fini avant que ces conditions soient satisfaites, en vertu de l'assertion  (iii) du  lemme~\ref{lem:induite} et  de la deuxième assertion du lemme~\ref{lem:petitangle}, respectivement. 
\end{proof}

\subsubsection{Argument dans un cadre un peu plus général} (Perturbation d'un cycle)
Pour une hauteur rationnelle $h^*$, le système est périodique puisque c'est un échange d'intervalles avec des translations rationelles. Considérons une partition suffisament fine pour que chaque intervalle soit envoyé exactement sur un intervalle et considérons le graphe correspondant. Ce graphe a des cycles ; chacun porte une (famille de) mesures invariantes. Imaginons d'abord qu'il n'y ait qu'un cycle et perturbons un peu la hauteur, i.e. considérons $h>h^*$ pas trop grand  ($h$ petit devant la longueur de l'intervalle sur la longueur du cycle). Essayons alors d'induire sur la réunion des intervalles de la partition appartenant au cycle. Le gros de cet ensemble est envoyé sur lui-même en respectant le cycle avec un petit décalage (de $2h$). Mais le bout (droit) de chaque intervalle est envoyé sur le bout gauche d'un autre intervalle de la partition. Ce bout suit ensuite la dynamique prescrite par le graphe (du rationnel, pour $h$ assez petit) et finit donc par revenir sur le cycle. Imaginons (cela n'est pas évident en général ; cela pourrait être dû aux symétries) que tous les bouts droits mettent le même temps (disons $k$) pour revenir au cycle. Ils arrivent alors décalés de $2kh$ ; poussons encore un peu pour attendre que chacun soit revenu dans l'image de l'intervalle de départ, ce qui nous fait $k'$ itérations. Il apparaît alors que l'induite sur la réunion des parties droites des intervalles du cycle (on a enlevé à gauche exactement $2k'h$) est essentiellement une rotation d'angle $2h$. Pour dire mieux, si on induit cette application sur un des morceaux (itérée $\ell$ fois), on a vraiment une rotation, d'angle $2 \ell h$.  Cet argument fonctionnerait pour certaines valeurs de $h^*$, mais nous ne voyons pas de manière de s'assurer qu'il fonctionnerait pour tout rationnel.


\begin{figure}
\includegraphics[width=150mm]{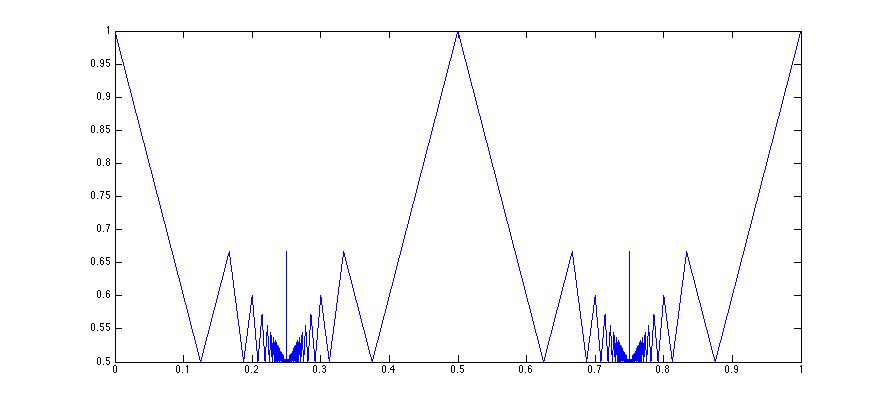}
  \caption{Représentation de la vitesse de fuite ($\lim H_n/n$) en fonction de $h$ pour $\tan \theta = 1/4$. Le graphe est tracé à partir d'une simulation. Elle est cohérente avec le résultat obtenu et ouvre des pistes pour des conjectures pour d'autres valeurs de $h$.}
  \label{fig:vitesse}
\end{figure}

\begin{figure}
\includegraphics[width=150mm]{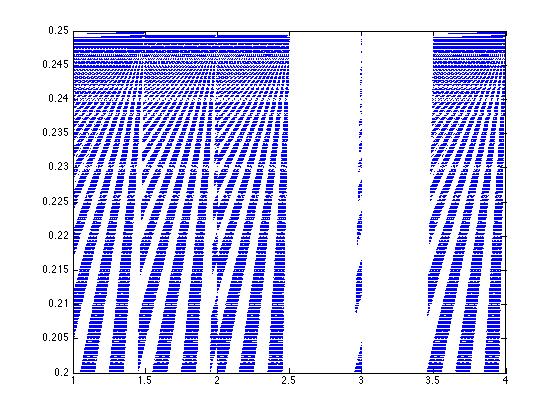}
  \caption{Vue rapprochée de la figure~\ref{fig:enslim} près de la valeur $h=1/4$, toujours pour $\tan \theta_0 = 1/4$.}
  \label{fig:presh=1/4}
\end{figure}


\subsection{Expression explicite de l'application de premier retour}
\label{sec:expressionvarphi}
Dans cette section, nous allons expliciter complètement l'application de premier retour dans la nouvelle paramétrisation proposée en section~\ref{sec:reparam}. Nous aurons besoin de quelques nouvelles notations. 

\subsubsection{Notations} On pose :  
$$\alpha = \frac{1}{2 K \tan{\theta_0}} \hbox{ et }  \beta(\xi, k) =   \sum_{i=0}^{k-1} \xi(i)   \frac{(i+1)}{K}-  \sum_{i=k+1}^{K-1}  \xi(i)  \frac{i}{K}, \hbox{ modulo } 1. $$
Pour tout $k \in \{0, 2K-1\}$, on définit  
$J_k =   [ \frac{k}{2K}, \frac{k+1}{2K}], \hbox{ et } J_k^h = J_k - h =   [ \frac{k}{2K}-h, \frac{k+1}{2K}-h].$
On observe qu'étant donnés $h>0$, $\gamma>0$ et $k$ il existe au plus deux valeurs $l$ et $l'$ de $\ell$ telles que $J_k \cap J_{\ell}^{\gamma}$ ne soit pas vides et qu'en outre $l'=l+1$.  
On note $\ell = \ell(\xi,k)$ le plus petit de ces deux entiers pour $h$ et $k$ donnés et $\gamma = \alpha + \beta(\xi, \tilde k)$.  On pose $\tilde \ell = \ell$  si $\ell$ est strictement inferieur à $K$ et $\tilde \ell = 2K - \ell -1$ sinon. Pour $\epsilon \in \{0, 1\}$, on définit  
$$I_k^{h,\epsilon}    =  J_k^h  \cap J_{\ell+\epsilon}^{h+\gamma}.$$

\vip

On définit une application $\varphi$  sur $\tore \times \cC_K^*$ en posant, pour chaque $(x,h,\xi)$ : soit $\epsilon \in \{0,1\}$, $k \in \{0, 2K-1\}$ tels que  $x \in I^{h,\epsilon}_k$.

\begin{equation}
\label{expressionvarphi}
\varphi(x,h,\xi) = \left\{ 
 \begin{array}{lllll}
(x+2h, &h, &\xi - \indiq_{\tilde k}&) & \hbox{ si } \xi(\tilde k) = 0 \hbox{ et } \sum_{i=0}^{K-1} \xi(i) > 1\\
 (x+2h, & h+\alpha, &1&) & \hbox{ si } \xi(\tilde k) = 0 \hbox{ et } \sum_{i=0}^{K-1} \xi(i) = 1\\
(x + 2h + 2\alpha +  \beta(\xi, \tilde k), &h+ \alpha,  & 1 - \indiq_{\{\widetilde{ \ell+\epsilon}\}}&) & \hbox{ si } \xi(\tilde k) = 1. 
\end{array}
\right.
\end{equation}

On vérifie  que si $x \in I_k^{h,\epsilon}$, alors, d'une part $x+h \in J_k$ et, d'autre part, $x+h + \alpha + \beta(\xi,\tilde k) \in J_{\ell + \epsilon}$. 
Pour $h$ fixé, on observe que  $\left\{  I^{h,0}_k, I^{h,1}_k, \,  0\leq k < 2K \right\}$ forme une partition du cercle en intervalles. 
On observe que $\ell (\xi,k) = k +\lfloor 2K (\alpha +\beta(\xi,k))\rfloor$ (modulo $2K$) et que si on note  $\gamma^* = \{2K(\alpha+ \beta(\xi,k)) \} / 2K$, on peut écrire 
\begin{eqnarray*}
 I^0_k &=&  [ \frac{k}{2K}-h, \frac{k+1}{2K} -h - \gamma^* ], \\
 I^1_k &=& [\frac{k+1}{2K} -h  -  \gamma^* , \frac{k+1}{2K}-h] , 
 \end{eqnarray*}


Enfin, on note  $\pi(\bfx, \eta, \theta) = x_1/2K$ si $\cos(\theta) \geq 0$  et $\pi(\bfx, \eta, \theta) = (2K-x_1)/2K$ si $\cos(\theta) <0$ modulo $1$, pour la première coordonnée dans la nouvelle paramétrisation.

\subsubsection{Résultat et preuve}
\begin{lemma}
\label{lem:expressionvarphi}
L'application $\varphi$ définie par (\ref{expressionvarphi}) est semi-conjuguée à l'application de premier retour au sens suivant : pour tout $(\bfx, \eta, \theta) \in B^*_{\bfh,\theta_0}$,  
$$ \varphi ( \Psi( \bf x, \eta,\theta) ) = \Psi( \phi_\bfh(\bf x, \eta,\theta) ).$$
\end{lemma}
\begin{proof} 
Tout d'abord, observons qu'il suffit de montrer le résultat pour $(\bfx,\eta,\theta)$ avec $H(\eta) = 0$ en vertu du Lemme \ref{lem:periode}, à condition de pouvoir le faire pour toute valeur de $\bfh$ dans $[0, 2K \tan\theta_0]$. 

\vip

\vip

Soit $\bfh \in [0, 2K \tan\theta_0]$. Soient  $\epsilon \in \{0,1\}$, $k \in \{0, 2K-1\}$ et $(\bfx,\eta,\theta) \in B_\bfh$ tels que  $\pi(\bfx,\eta,\theta) \in I^{h,\epsilon}_k$. On note $(x,h,\xi ) = \Psi(\bfx,\eta,\theta)$.  
On va diviser la trajectoire jusqu'au premier retour à la base  en 4 parties distinctes : soient $t_0$ l'instant où la trajectoire (montante) atteint l'ordonnée $0$, $t_1$ l'instant où la trajectoire rebondit sur une paroi horizontale et change de sens et $t_2$  l'instant où la trajectoire passe à l'ordonnée $0$ en descendant. On a toujours $t_0=\bfh/\sin( \theta_0)$ vue la vitesse verticale. Observons tout de suite que si la bille frappe la première brique sur une paroi horizontale, $t_0=t_1=t_2$~; autrement, $t_1= t_0 +  1/ \sin( \theta_0)$ et $t_2= t_1 +  1/ \sin( \theta_0)$.  

Le déplacement horizontal est relativement simple à analyser. En effet, la symetrie introduite fait que hormis lors des rebonds sur des briques, vu dans les nouveaux paramètres,  le déplacement horizontal est toujours proportionnel au temps. La vitesse du flot est donnée par $\cos \theta_t$ et dépend donc de son signe. 
Si on excepte les sauts qui se produisent aux moments des reflexions sur des parois verticales, la  quantité $\pi(\Phi_t(\bfx,\eta,\theta))$ évolue tout au long de la trajectoire à vitesse constante ($\frac{1}{2K} \cos \theta_0$). Mais sur les parois verticales fixes,  $x_1=0$ et $x_1=K$ l'évolution reste continue (par construction !) puisque $\pi((K,y),\theta)=\pi((K,y),-\theta) = 1/2$ alors que $\pi((0,y),\theta) = 0$ et $\pi((0,-\theta) = 1$ sont égaux modulo $1$.  
Pour obtenir le deplacement horizontal pendant l'aller-retour, il restera donc à ajouter les sauts qui se produisent lors des rebonds sur les parois verticales des briques (uniquement en fait entre $t_0$ et $t_1$). A ces instants on passe de $((x_1,y),\eta, \theta)$ à $((x_1,y), \eta',-\theta)$ ; et donc un rebond contre une paroi verticale à l'abscisse $x_1$ produit  un saut de $\pi((x_1,y),\eta, -\theta) - \pi((x_1,y), \eta,\theta) =  1-2 \pi((x_1,y),\eta, \theta)$. Nous y reviendrons.

L'analyse du déplacement horizontal entre les instants $0$ et $t_0$ montre que si $\pi(\bfx,\eta,\theta) \in I^h_k$, alors, $\pi(\Phi^h_{t_0}(\bfx,\eta,\theta)) \in [ k, k+1]/2K$ à l'instant $t_0$ où on atteint la hauteur $0$. On distingue alors les même cas que dans la preuve de la proposition~\ref{prop:conjug} : 
\vip

\noindent
(i) :  si $\xi(\tilde k) =1$, la brique $\square_{(\tilde k,0)}$est occupée et on rebondit sur sa face inferieure avant de revenir à la base. La configuration $\eta$ est changée en $\eta^{\square_{(\tilde k,0)}}$. Le déplacement horizontal ne dépend que de la durée de la trajectoire. Modulo $1$, il est de $\gamma(h) =2 \times  \frac{1}{2K} \times \frac{\bfh}{\sin( \theta_0)} \times \cos(\theta_0) =  2h$ (aller et retour ; pas de rebonds sur des parois verticales non fixes) : 
\begin{eqnarray*}
\pi(\phi(\bfx,\eta,\theta)) &=& \pi(\bfx,\eta,\theta) + 2h = x + 2h \\
\end{eqnarray*}
Si la brique touchée n'était pas la dernière brique à cette hauteur,  i.e., si $\sum_{i=0}^{K-1} \xi(i) > 1$,  la hauteur est inchangée et  la nouvelle configuration est $\xi' = \xi - \indiq_{\{ \tilde k\}}$ : 
\begin{eqnarray*}
H(\phi(\bfx,\eta,\theta)) &=& 0\\
\xi(\phi(\bfx,\eta,\theta)) &= & \xi-\indiq_{\tilde k}.
\end{eqnarray*}
Si, en revanche,  la brique touchée était la dernière brique à cette hauteur,  i.e., si $\sum_{i=0}^{K-1} \xi(i) = 1$, la hauteur est incrémentée d'une unité  et la nouvelle configuration est remplie :  $\xi' \equiv 1$ ; Ainsi, dans ce cas, 
\begin{eqnarray*}
H(\phi(\bfx,\eta,\theta)) &=& 1\\
\xi(\phi(\bfx,\eta,\theta)) &= & \indiq.
\end{eqnarray*}

\vip

\noindent
(ii) : si $\xi(\tilde k) =0$, vue la condition $ \theta_0 < \theta^*(K)$, on va heurter toutes les briques de l'étage $0$ avant d'aller rebondir au niveau $1$. La hauteur sera systématiquement incrémentée de $1$.  Pour déterminer le déplacement horizontal et la configuration finale, il faut préciser ce qui se passe.  Entre $0$ et $t_0$ puis au delà de $t_1$, la trajectoire ne rencontre pas d'obstacles et le déplacement horizontal est proportionnel à la durée. La trajectoire entre $t_0$ et $t_1$, c'est-à-dire la trajectoire montante entre les hauteurs $0$ et $1$ est plus délicate à traiter. Ainsi, on dépasse la hauteur $0$ en entrant dans la case "vide" $\tilde k$.  On va  heurter toutes les briques restant dans la configuration à hauteur $0$, alternativement d'un coté puis de l'autre (jusqu' à ce que l'un des deux cotés soit vide) ; ensuite on heurte alternativement les briques restant d'un coté et le bord vertical opposé. Comme nous l'avons remarqué ci dessus, la projection du déplacement horizontal augmente de manière constante avec le temps {\it excepté au moment des rebonds}.  Aux rebonds avec des faces verticales, la projection du  déplacement fait un saut : si la projection de la coordonnée horizontale de la face verticale heurtée est $x$, on "saute" en $1-x$. Ce qui correspond à un déplacement de $-2x$ (modulo $1$). Pour connaître le déplacement horizontal, il faut donc se donner les abscisses des parois verticales sur lesquelles la trajectoire effectue un rebond. 
Pendant cette partie de la trajectoire, la bille heurte la face de gauche (en arrivant avec une vitesse horizontale  $\cos \theta >0$) des briques $\{ \square_{(i,0)} : \xi(i) =1, i > \tilde k \}$ situées à droite de  $\square_{(\tilde k,0)}$ et les faces de droite (en arrivant avec $\cos \theta <0$) des briques situées à gauche de $\square_{(\tilde k,0)}$. La face de gauche et la face de droite de la brique $(i,0)$ ont respectivement  pour projection $i/2K$ ou $1-i/2K$ et $(i+1)/2K$ ou $1-(i+1)/2K$ (soit $\pm$ modulo $1$). Insistons sur le fait que le déplacement horizontal  dépend du sens dans lequel on rebondit sur la paroi ; mais ce sens est déterminé. En revanche, le déplacement total  ne dépend pas de l'ordre dans lequel ces parois sont touchées. 

Ainsi, entre $t_0$ et $t_1$, le déplacement horizontal est la somme du déplacement continu qui vaut $\alpha = 1/(K \tan \theta)$  pendant la durée $t_1-t_0$ et du déplacement dû aux sauts qui vaut : 
$$ \beta(\xi, \tilde k)=    \sum_{i=0}^{k-1} \xi(i)   2\frac{i+1 }{2K} -  \sum_{i=k+1}^{K-1}  \xi(i)  2 \frac{i}{2K}. 
$$   
Il est maintenant facile de déterminer la brique à hauteur $1$ qui est frappée. En effet, si $\epsilon=0$, alors $\pi(\Phi^h_{t_1}(\bfx,\eta,\theta)) = \pi(\Phi^h_{t_0}(\bfx,\eta,\theta)) + \alpha + \beta(\xi, \tilde k) \in [\ell/2K,(\ell+1)/2K]$ (par définition de $\ell$) de telle sorte que l'abscisse de la bille en $t_1$ est sur la face inferieure de la brique en $\tilde \ell$ ; alors que si $\epsilon=1$, alors $\pi(\Phi^h_{t_1}(\bfx,\eta,\theta)) \in [(\ell+1)/2K,(\ell+2)/2K]$ de telle sorte que l'abscisse de la bille en $t_1$ est sur la face inferieure de la brique en $\widetilde{ \ell +1}$.  Ainsi, dans les deux cas, la brique frappée est la brique $\widetilde{ \ell +\epsilon}$. 

On synthétise cette information sous la forme,  
\begin{eqnarray*}
\pi(\phi(\bfx,\eta,\theta)) &=& x 
+ h + \alpha +  \beta(\xi,\tilde k) +\alpha + h\\
H(\phi(\bfx,\eta,\theta)) &=&1\\
\xi(\phi(\bfx,\eta,\theta)) &= & 1 - \indiq_{\{\widetilde{\ell(\xi,k)+\epsilon\}}}
\end{eqnarray*}

\vip
Il ne reste plus qu'à écrire cela dans les nouveaux paramètres et à vérifier que cela passe bien au quotient ; c'est ce qu'affirme le Lemme \ref{quotient}. 
\end{proof}

\subsubsection{Cas $K=2$} Pour illustrer la forme de cette application,  nous allons l'expliciter complètement lorsque $K=2$. 
Fixons maintenant $K=2$.  Posons  $\gamma(h)=2h$ et $\beta(h) = 2h + 1/2+ 2 \alpha$. Notons $I_g = J_0 \cup J_3$ et $I_d = J_1 \cup J_2$. Il ressort que : 
 \begin{eqnarray*}
\varphi(x,h,(1,1)) &=& \left\{ \begin{array}{lll} (x+\gamma(h), h, (0,1)),  & \hbox{     }  \hbox{     }& \hbox{ si } x+h  \in I_g \\  (x+\gamma(h), h, (1,0)),  && \hbox{ si } x +h \in I_d  \end{array} \right. \\
\varphi(x,h,(0,1)) &=&  \left\{ \begin{array}{ll}  (x+\beta(h), h+\alpha, (0,1) ) & \hbox{ si } x+h \in I_g \hbox{ et } x+1/2 + h + \alpha \in I_g  \\
(x+\beta(h), h+\alpha, (1,0) ) & \hbox{ si } x +h \in I_g  \hbox{ et } x+1/2+h+\alpha \in I_d \\
(x+\gamma(h), h+\alpha, (1,1)),  & \hbox{ si } x+h \in I_d  \end{array} \right. \\
\varphi(x,h,(1,0)) &=& \left\{ \begin{array}{ll} (x+\gamma(h), h+\alpha, (1,1)),  & \hbox{ si } x +h \in I_g \\   
(x+\beta(h), h+\alpha, (0,1) ) & \hbox{ si } x+h \in I_d \hbox{ et } x+1/2 +h + \alpha  \in I_g  \\
(x+\beta(h), h+\alpha, (1,0) ) & \hbox{ si } x +h \in I_d  \hbox{ et } x+1/2+h+\alpha \in I_d  \end{array} \right. 
\end{eqnarray*}
Observons aussi que, lorsque $K=2$, l'ensemble $A_+ \cup A_{++}$ s'écrit : 
$$ A_+ \cup A_{++} =  \tore  \times \{ (1,0), (0,1)\}.$$

\section{Domaine $\rr^2$}
\label{sec:zdeux}
Nous allons  nous intéresser au casse-briques $\Phi$ dans $\Delta = \rr^2$ muni de la famille d'obstacles $\cO = \{ \carre_{{\bf z}}, {\bf z} \in \zz^2 \setminus \{0,0\} \}$.  
L'ensemble des positions initiales admissibles est naturellement  le carré $[0,1]^2$. Les résultats de la section~\ref{sec:definition} s'applique naturellement à ce contexte. Comme précédemment, commençons par observer que : 
\begin{itemize}
\item Rappelons que l'ensemble des configurations bien remplies est stable~: en effet, le nombre de briques cassées par unité de temps est fini.  Observons  que l'ensemble des configurations rendant le domaine connexe est stable lui aussi. 
\item Pour les même raisons que dans le domaine restreint, partant d'un état initial $(\bfx_0,\eta_0,\theta_0)$ régulier,  l'angle reste dans l'ensemble $\{\theta_0, \pi-\theta_0,-\theta_0,\pi+\theta_0\}$. 
\item Mentionnons les symétries évidentes de ce système par rapport aux axes de coordonnées et par rapport aux diagonales. 
\item Observons enfin que si au lieu de briques carrées $1 \times 1$ on avait pris des briques rectangulaires $L \times \ell$, on pourrait se ramener au système avec briques carrées en changeant l'angle initial de telle manière que $L \tan{\theta} = \ell \tan{\theta'}$.  
\end{itemize}

\subsection{Un facteur}
Modulo $\zz^2$, ce système est proche d'un billard carré. Pour préciser ce point, nous allons   faire un quotient sur le tore (en oubliant configuration et direction de la vitesse). 
Considérons l'application $\zeta : \adm \to \tore \times \cercle$ définie par $\zeta(\bfx,\theta, \eta) = (\{\bfx\},v(\theta))$ (où $\{\bfx\} = (\{x_1\},\{x_2\})$ désigne le vecteur composé des parties fractionnaires des coordonnées) et $\overline{\zeta}$ qui à $(\bfx, \theta, \eta)$ associe la classe de $\zeta(\bfx, \theta, \eta)$ dans le quotient par les deux  relations :  $((x_1,x_2),(v_1,v_2))\equiv ((x_1,1-x_2), (v_1,-v_2))$ et $((x_1,x_2),(v_1,v_2))\equiv ((1-x_1,x_2), (-v_1,v_2))$. Disons que $\overline{\zeta}(\bfx, \theta, \eta)$ est le seul représentant avec les coordonnées de vitesse positives (et on oublie la vitesse qui ne "change" pas). C'est un point dans le tore avec une vitesse positive. L'observation est que pour tout $t>0$, 
$$\overline{\zeta} (\Phi_t(\bfx, \theta, \eta)) = \overline{\zeta} (\bfx, \theta, \eta) + t (|v_1|,|v_2|), \, \hbox{ modulo } \tore,$$ 
c'est-à-dire que $\Phi_t$ est semi-conjugué au flot de direction $(|v_1|,|v_2|)$ sur le tore $\tore$. C'est trivial pour les mouvements dans chaque case. Au bord des cases, ce qui arrive au flot dépend de la configuration dans la case adjacente, mais le quotient prend en compte tous les cas de figures qui peuvent se produire (traversée, rebond sur une arête horizontale, rebond sur une arête verticale) si bien que l'évolution ne dépend pas de la configuration. 

\vip
Lorsque la tangente de l'angle est rationnelle, ce système dynamique est conjugué à un automate cellulaire ; en effet, l'ensemble des angles reste fini, mais,  en vertu du point précédent, l'ensembles des positions  admissibles au bord d'une case est alors lui aussi fini. Un état est donc décrit par une configuration et une donnée finie. En outre, il est possible, à partir de ces éléments de déterminer l'état du système (au prochain instant où il atteint le bord d'une case) localement, i.e. en ne regardant la configuration que dans un voisinage de la case occupée. Cette remarque ne nous ayant pas permis d'obtenir de résultats, nous ne la développerons pas plus avant. 

\subsection{Périodicité relative}
On cherche ici à caractériser des comportements particulièrement simples. On ne peut espérer trouver d'orbites périodiques (puisque le nombre de briques ne fait que décroitre). Le cas qui semble le plus simple est celui d'un angle $\theta_0 = \pi/2$ (ou $\theta_0=0$). Dans ce cas, la bille "creuse" alternativement une brique vers le haut et une vers le bas, formant une bande verticale de largeur $1$ à l'abscisse initiale. Si on part d'une configuration non triviale, il reste une tache au depart mais cela ne change rien au comportement asymptotique. Notre idée est de definir ce que seraient des comportements pratiquement aussi simples dans d'autres directions. 
Nous proposons pour aborder ce problème quelques définition inspirées par ce cas trivial.

\vip

Commençons par un peu de formalisme. Partant d'un état $(\bfx_0, \eta_0, \theta_0)$,  nous noterons $\eta_n$ la configuration après le $n^{{ieme}}$ rebond et $\Delta_n = \Delta^{\eta_n}$ la forme de l'ensemble des briques enlevées à ce moment. Il est clair ici que si on part d'une configuration bien remplie, le nombre de brique décroit indéfiniment. On note $\Delta_\infty = \lim_{n \to \infty} \Delta_n$. On appelle {\em bande}  (respectivement {\em demie bande}) de pente $\alpha$, d'origine $\bfx$, de largeur $\delta$,  l'ensemble des points à distance  inferieure à $\delta$ de la  droite (resp. demie droite) partant de $\bfx$ et  formant un angle  $\alpha$ avec l'horizontale.  

\begin{definition}
Une orbite du casse-briques est dite {\em directionnelle} si $\Delta_\infty$ est contenue dans une bande. 
\end{definition}
On peut être tenté de penser que la direction de la droite est soit $\theta_0$, soit $\pi/2- \theta_0$. C'est faux. Nous allons voir qu'il existe un cas où $\theta_0 = \arctan{3}$ et où la bande est de direction  $\alpha = 0$~! Mais dans ce cas, on peut inclure la forme limite dans une demie bande. Si la forme limite contient elle même une bande, cela implique qu'on fasse des allers et retours dans la bande et alors, les seules directions possibles sont effectivement des directions prises par la bille : $\theta_0$ et $\pi/2 - \theta_0$. Mais ce phénomène  suggère les définitions suivantes : 
\begin{definition}
On dit qu'une orbite du casse-briques {\em creuse dans la direction $\alpha$} si il existe une demie bande de pente $\alpha$ contenue dans $\Delta_\infty$. 
\end{definition}
\begin{definition}
On dit qu'elle {\em ne fait que creuser} s'il existe un nombre fini de demie bandes  (de pentes $\alpha_1, \ldots, \alpha_K$) qui recouvrent la forme limite. 
\end{definition}
On va montrer  dans ce cas, on peut se contenter de prendre $K=1$ ou $K=2$ et que si $K=2$, alors $\alpha_2 = \alpha_1+\pi$ et l'orbite est directionelle. 
\begin{lemma}
Si une orbite ne fait que creuser, alors on peut recouvrir la forme limite avec une demie bande ou une bande ; autrement dit, $K=1$ ou $K=2$ suffisent.  
\end{lemma}
\begin{proof}
Imaginons que la forme limite soit couverte par un nombre fini de demies bandes. Choisissons $R_1$ assez grand pour que les parties des demies bandes hors du disque $D_1$ de rayon $R_1$ soient disjointes et notons $\alpha_0$ la plus petite distance angulaire entre deux demies bandes. S'il finit par arriver que la bille se perde dans l'une des demies bandes, alors, on aurait pu recouvrir la forme limite avec une seule demie bande. Dans l'alternative,  considérons  une demie  bande dans laquelle la bille s'aventure arbitrairement loin et revient dans le disque $D_1$ ; après un certain temps, elle ne rencontre plus de briques dans $D_1$ et donc ne change pas de direction en le traversant.  Fixons $\epsilon$ positif. Choisissons $R_2$ assez grand par rapport à la largeur de la bande pour être sûrs que si la bille arrive depuis une distance $R_2$ (sans toucher de brique) au "bout" de la demie bande $k$, son angle en traversant  le disque de rayon $R_1$ soit assez proche de $\alpha_k$ pour qu'on soit sûr que son point d'impact de l'autre côté de $D_1$  soit contrôlé à $\epsilon$ près. Si $\epsilon$ est assez petit par rapport à  $\alpha_0$, la bille ne peut alors entrer que dans l'une des autres demies bandes. Comme la bille doit alors atteindre le "bout" ce cette demie bande, l'angle de celle-ci doit être arbitrairement proche de $\alpha_k + \pi$. En prenant $\epsilon$ arbitrairement petit, on conclut que seules deux demies bandes opposées suffisent à couvrir la forme limite. 
\end{proof}
Pour chaque demie bande de pente $\alpha$  dans laquelle l'orbite creuse, on considère la suite $(\bfz^\alpha_n)$ des briques détruites dans cette demie bande et la suite de différences $\partial \bfz^\alpha_n = \bfz^\alpha_n - \bfz^\alpha_{n-1}$.
\begin{definition}
L'orbite est dite {\em relativement périodique} si les suites $(\partial \bfz_n^{\alpha_k})$ sont ultimement périodiques pour tous $k=1,\ldots, K$.   
\end{definition}
Nous allons exhiber des orbites relativement périodiques dans les cas : $\tan{\theta_0} = 1$, $\tan{\theta_0} = 2$ et $\tan{\theta_0} = 3$. Dans les trois cas, ces orbites sont directionelles, de pentes soit $\alpha = \theta_0$, soit $\alpha = \pi/2 - \theta_0$. Dans le dernier cas, $\theta_0 = \arctan{3}$, nous trouvons aussi une orbite qui creuse  dans la direction $\alpha = 0$ et uniquement dans cette direction (i.e. $K = 1$ ; nous disons qu'elle est {\em auto-creusante}). Ces exemples sont illustrés par les figures~\ref{fig:atan1}-\ref{fig:atan3_2} avec différentes conditions initiales. Nous soupçonnons qu'il existe  des orbites relativement périodiques pour beaucoup d'autres directions  mais n'avons su en mettre une seule autre en évidence ; par exemple même pas pour $\arctan{4}$ ! Ce phénomène nous paraît frappant ; il est probablement dû à notre incompétence. 

\vip

Nous n'avons pas non plus su montrer le résultat suivant qui paraît au premier abord relativement intuitif, mais qui pourrait soulever des difficultés comparables à celles posées par la conjecture de la fourmi de Langton~\cite{fou}~:

\begin{conjecture}
Quelle que soit la condition initiale  $(\bfx_0, \eta_0, \theta_0)$, si  $\eta_0$ est bien remplie et si $\tan{\theta_0} = 1$, alors son orbite est relativement périodique. 
\end{conjecture}

Dans le cas relativement périodique, il est possible de définir un (ou plusieurs) {\em motifs} tels que, en dehors d'une boule finie la configuration limite $\Delta_\infty$ soit exactement une réunion de translatés de ces motifs. Plus précisemment, pour chaque direction $\alpha$, on définit la période $p^\alpha$ de $\partial \bfz^\alpha$ et sa pré-période $n_0^\alpha$, le vecteur ${\bf v}^\alpha = \bfz^\alpha_{n_0^\alpha+p^\alpha} - \bfz^\alpha_{n^\alpha_0} = \sum_{i=n_0}^{n_0^\alpha+ p^\alpha} \partial \bfz^\alpha_i$  et le motif $P^\alpha = \bigcup_{i=n^\alpha_0}^{n^\alpha_0+ p^\alpha} \square_{\bfz^\alpha_i}$. Alors 
$$\Delta_\infty \setminus \left(   \bigcup_{k=1}^{K} \bigcup_{n \geq 0} P^{\alpha_k} + n {\bf v}^{\alpha_k} \right) $$
est un ensemble fini. Cela entraine naturellement que les directions $\alpha_k$ (qui est aussi la direction de ${\bf v}^{\alpha_k}$) sont toutes rationnelles. Nous soupçonnons que, plus généralement,  si l'orbite ne fait que creuser, alors la direction doit être rationnelle ; mais nous n'avons pas d'argument convaincant en fait. 

\vip

Il est aussi assez facile de voir que dans le cas $K=2$,  la direction est nécessairement celle de l'angle parce qu'on fait des trajets de plus en plus long sans toucher de brique sur une largeur finie. Mais rien ne garantit que les seuls motifs possibles sont ceux que nous avons observé. On peut se demander en particulier s'il existe des orbites auto-creusantes pour $\theta =\arctan{1}$ et $\theta = \arctan{2}$ ? Toute orbite directionelle est-elle relativement périodique ? Tout angle de tangente rationnelle donne-t-il lieu à des orbites relativement périodique ? Si c'est le cas, nos expérimentations semblent indiquer que la prépériode période peut être très longue.

\subsection{Exemples}
\label{sec:exemples}

\begin{figure}
\includegraphics[width=65mm]{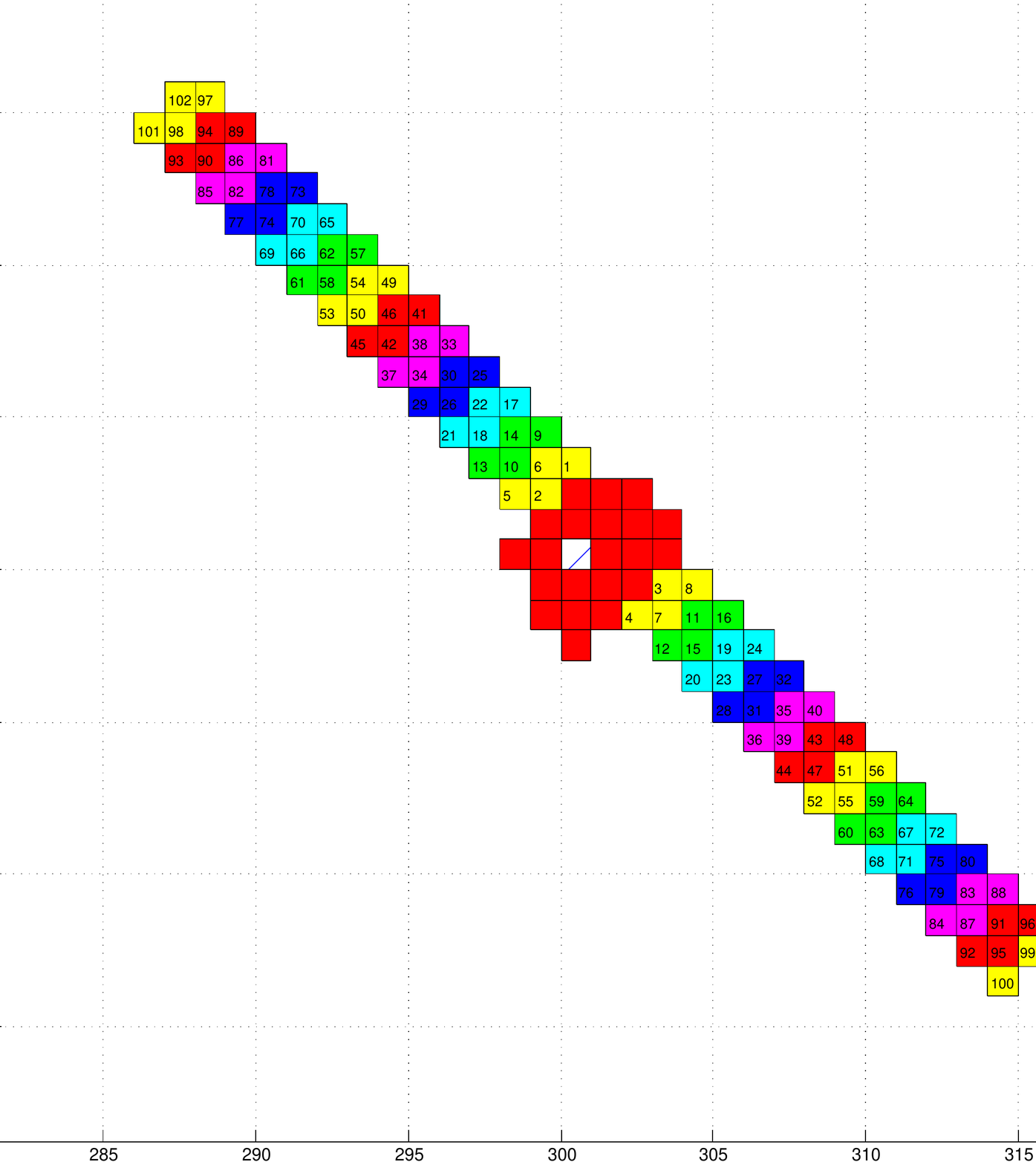}
\hspace{.2cm}
\includegraphics[width=65mm]{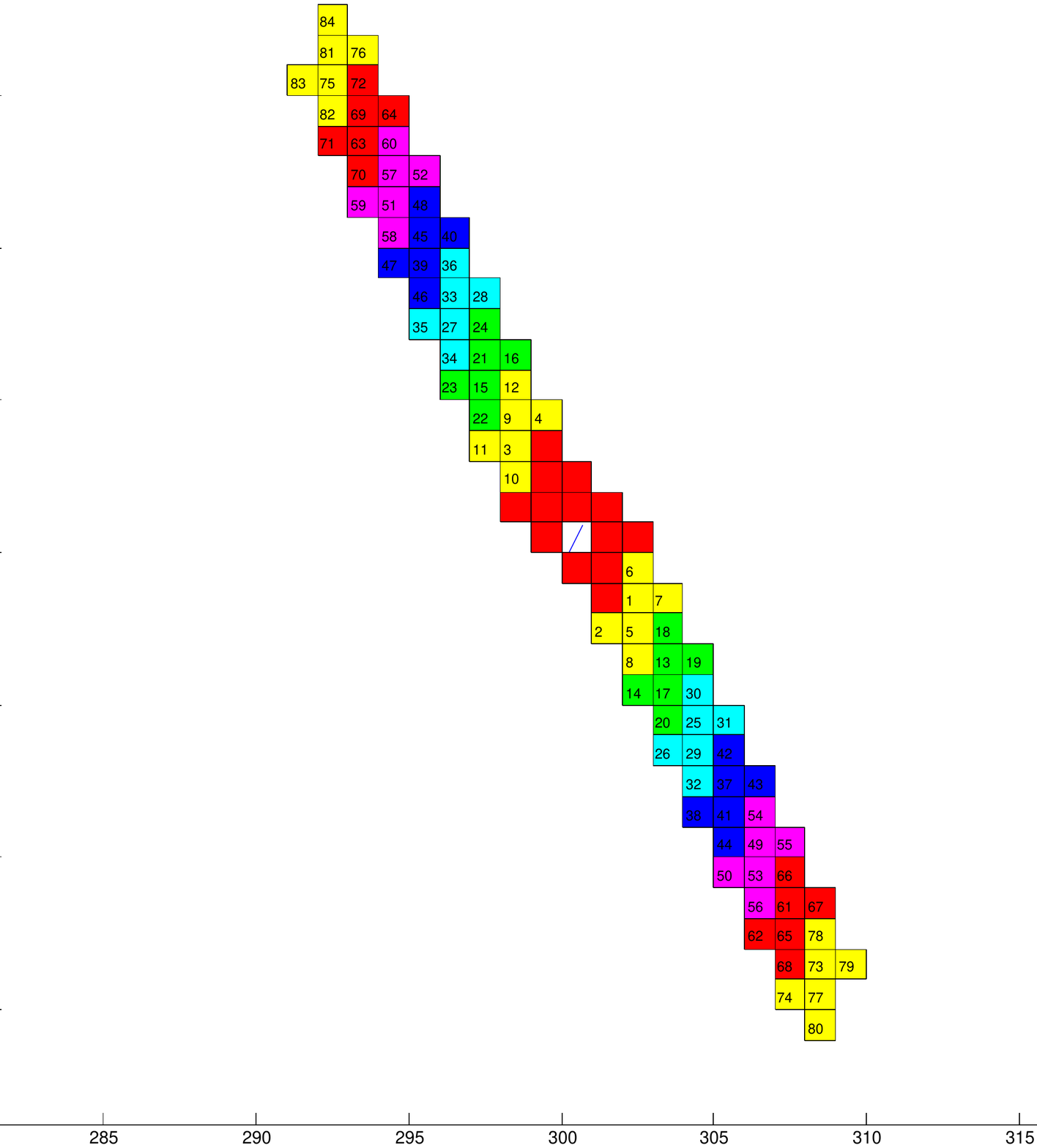}
  \caption{ $\tan \theta_0 = 1$ et $\tan \theta_0 = 2$.}
  \label{fig:atan1}
\end{figure}
\begin{figure}
\includegraphics[width=65mm]{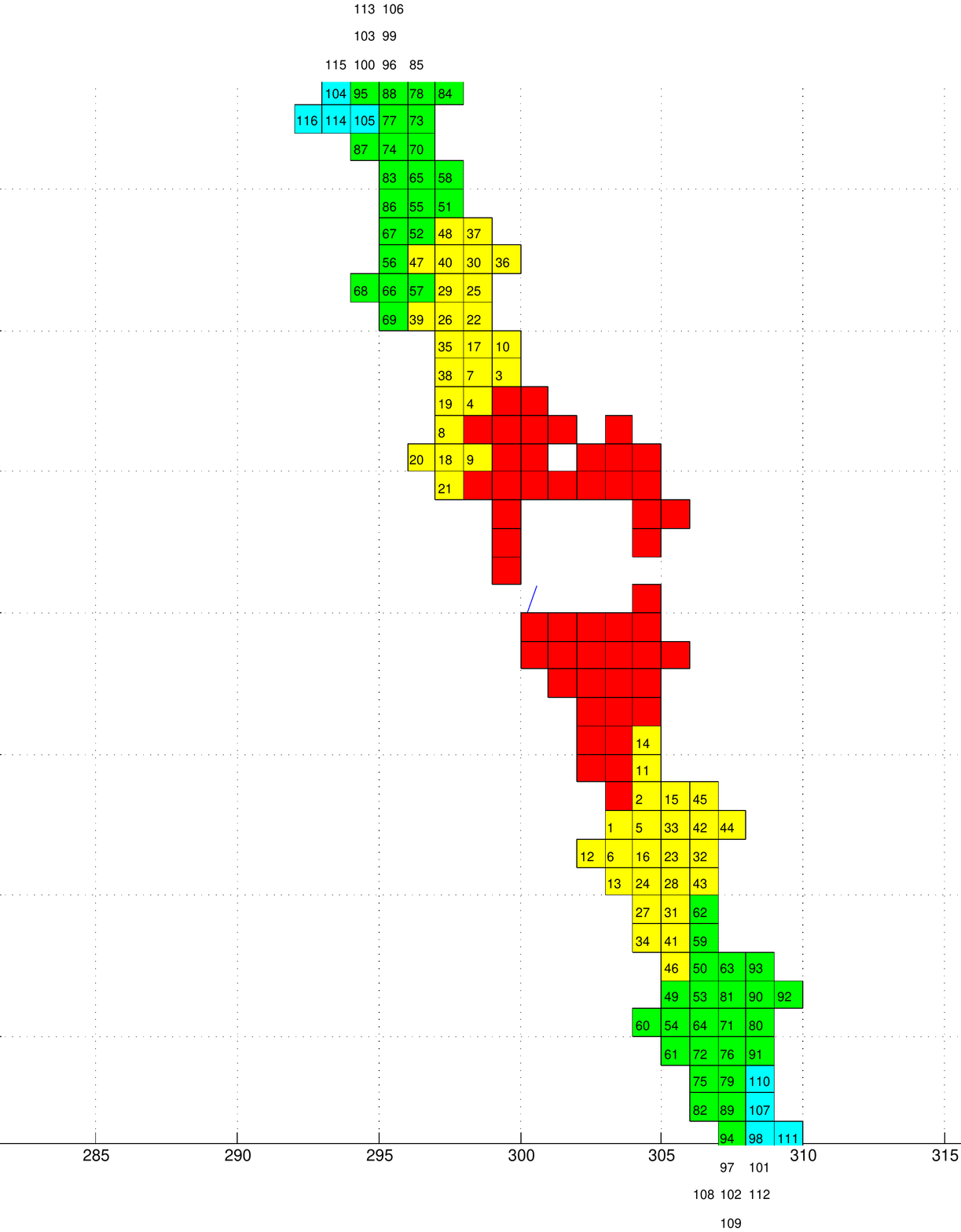}
\hspace{.2cm}
\includegraphics[width=65mm]{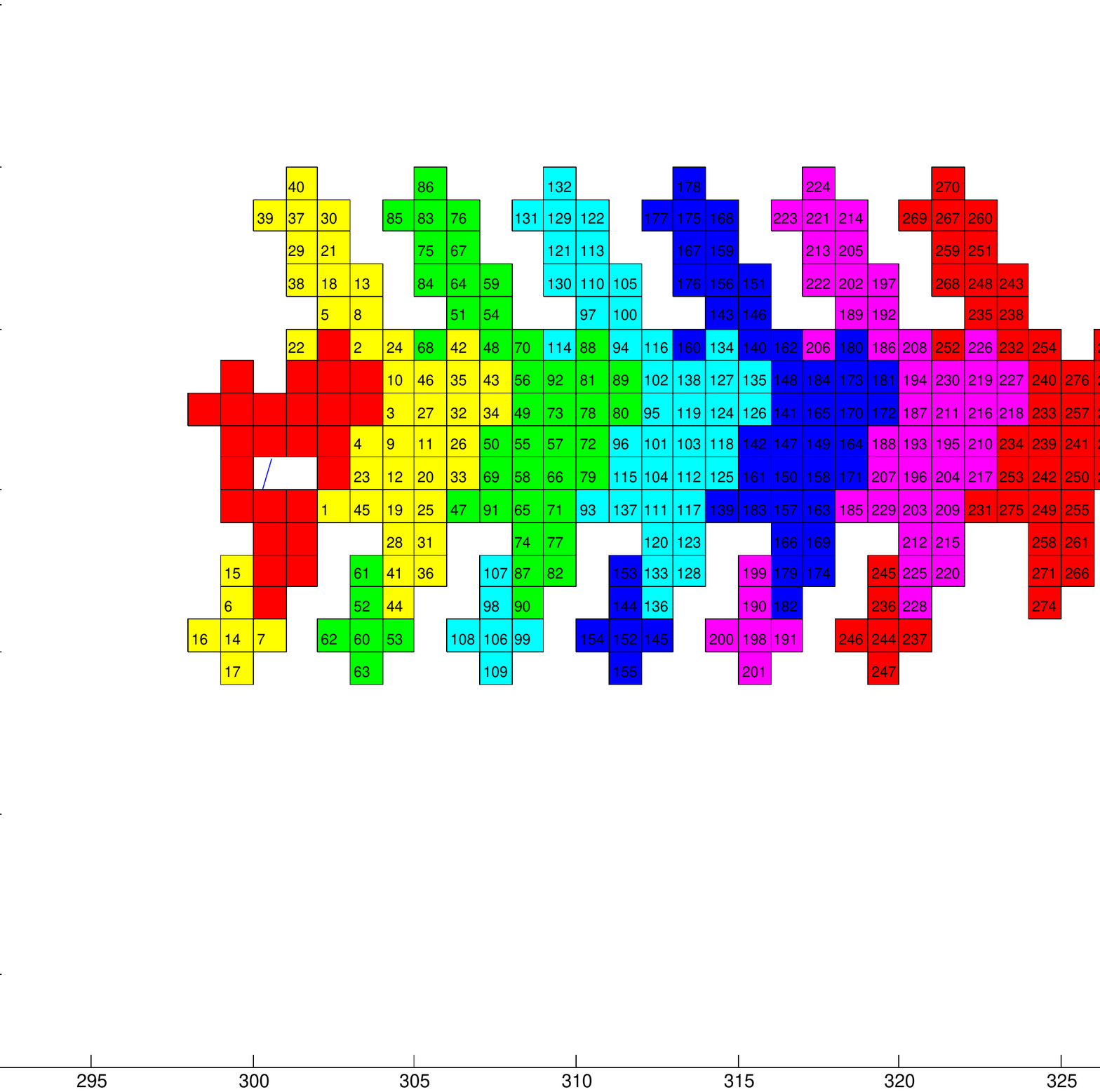}
  \caption{ $\tan \theta_0 = 3$.}
  \label{fig:atan3_2}
\end{figure}

Nous décrivons ici les seuls exemples d'orbites relativement périodiques que nous avons pu observer. Elles sont décrites par des figures ; la numérotation des briques permet de reconstituer les trajectoires une fois le régime périodique atteint. Observons qu'il y a une part d'arbitraire dans le choix du début de la période qui détermine la forme des motifs. 
$$
\begin{array}{|c|c|c|}
\hline
\hbox{Angle} & \hbox{Branches} & \hbox{Périodes} \\
\hline
\tan{\theta} = 1 & 2 & 4,4\\
\hline
\tan{\theta} = 2 & 2 & 6,6\\
\hline
\tan{\theta} = 3 & 2 & 24, 24\\
\hline
\tan{\theta} = 3 & 1 &46 \\
\hline
\end{array}
$$

\vip 

Dans les figures présentées, les configurations initiales sont choisies de manière à ce que les prépériodes ne soient pas trop longues. Mais nos expérimentations n'ont montré que ces exemples et leurs variantes symétriques pour toutes les configurations initiales essayées. Nous n'avons en revanche pas trouvé d'autres angles produisant des orbites relativement périodiques malgré de nombreuses et longues tentatives, par exemple avec $\theta = \arctan{4}$. Il semble que Yann Jullian [communication personnelle] ait mis en évidence  de nouvelles orbites relativement périodiques. 


%
%
%


\def\refname{Références}

\end{document}